\documentclass[11pt,leqno]{article}
\usepackage[T1]{fontenc}
\usepackage{amsmath}%
\usepackage{amsthm}
\usepackage{amsxtra}%
\usepackage{amsfonts}%
\usepackage{amssymb}%
\usepackage{color,hyperref}
\usepackage{graphicx}
\usepackage{cite}
\usepackage{subfigure}
\hypersetup{colorlinks,breaklinks,
             linkcolor=blue,urlcolor=blue,
           anchorcolor=blue,citecolor=blue}

\usepackage{geometry}

\newcommand{\dist}{\text{dist}}
\renewcommand{\O}{{\mathcal O}}
\renewcommand{\bar}[1]{{\overline{#1}}}
\newcommand{\F}{{\mathcal F}}
\renewcommand{\P}{{\mathbb P}}
\newcommand{\X}{{\mathcal X}}
\newcommand{\W}{{\mathcal W}}
\newcommand{\R}{\mathbb{R}}
\newcommand{\T}{\mathbb{T}}
\newcommand{\G}{{\mathcal G}}
\renewcommand{\div}{\text{div}}
\newcommand{\N}{\mathbb{N}}
\newcommand{\D}{\nabla}

\newcommand{\E}{\mathbb{E}}

\newcommand{\Z}{\mathbb{Z}}

\newcommand{\eps}{\varepsilon}
\newcommand{\dJ}{\Delta_p^G}

\renewcommand{\phi}{\varphi}

\newtheorem{theorem}{Theorem}
\newtheorem*{theorem*}{Theorem}
\theoremstyle{plain}

\newtheorem{definition}{Definition}

\newtheorem{lemma}{Lemma}
\newtheorem{proposition}{Proposition}
\newtheorem{remark}{Remark}

\title{The game theoretic $p$-Laplacian and semi-supervised learning with few labels}
\author{Jeff Calder\thanks{Department of Mathematics, University of Minnesota. ({\tt jcalder@umn.edu})}}
\begin{document} 
\maketitle
\begin{abstract}
We study the game theoretic  $p$-Laplacian for semi-supervised learning on graphs, and show that it is well-posed in the limit of finite labeled data and infinite unlabeled data. In particular, we show that the continuum limit of graph-based semi-supervised learning with the game theoretic  $p$-Laplacian is a weighted version of the continuous $p$-Laplace equation. We also prove that solutions to the graph $p$-Laplace equation are approximately H\"older continuous with high probability. Our proof uses the viscosity solution machinery and the maximum principle on a graph. 
\end{abstract}

\section{Introduction}
Large scale regression and classification problems are prominent today in data science and machine learning. In a typical problem, we are given a large collection of data, some of which is labeled, and our task is to extend the labels to the larger data set in some meaningful way. \emph{Fully supervised} learning algorithms make use of only the labeled data. They may, for example, learn a parametrized function $f$ on the ambient space that maximally agrees with the labeled data. Depending on the application, one might use support vector machines, or neural networks, for example. Fully supervised algorithms are generally only successful when there is a sufficiently diverse collection of labeled data to learn from. In many applications, labeled data requires human annotation, often by experts, and can be difficult and expensive to obtain.  

On the other hand, due to the generally increasing availability of data, unlabeled data is essentially free. As such, there has recently been significant interest in \emph{semi-supervised} learning algorithms that make use of both the labeled data, as well as the geometric or topological properties of the unlabeled data to achieve superior performance. Semi-supervised learning is used in many applications where a large amount of data is available, but only a very small subset is labeled. Examples include classification of medical images, natural language parsing, text recognition, website classification, protein sequencing to structure problems, and many others \cite{ssl}.

We consider semi-supervised learning on graphs. Here, we are given a weighted graph $G=(\X,\W)$, where $\X$ are the vertices, and $\W=\{w_{xy}\}_{x,y\in \X}$ are nonnegative edge weights. The weights are typically chosen so that $w_{xy}\approx 1$ when $x$ and $y$ are close, and $w_{xy}\approx 0$ when $x$ and $y$ are far apart (see Eq.~\eqref{eq:weights}). The labeled data form a subset $\O \subset \X$ of the vertices called the \emph{observation set}.  Given an unknown real-valued function $g:\X \to \R$, and observations $g(x)$ for $x \in \O$, the goal of semi-supervised learning is to make predictions about the function $g$ at the remaining vertices $x \in \X\setminus \O$. 

Since there are infinitely many ways to extend the labels, the problem is ill-posed without some further assumptions. It is now standard to make the \emph{semi-supervised smoothness assumption}, which says that we should extend the labeled data in the ``smoothest'' way possible, and the degree of smoothness should be locally proportional to the density of the graph.  A widely used approach to the smoothness assumption is Laplacian regularization, which leads to the minimization problem

\begin{equation}\label{eq:L2}
\min_{u:\X\to \R} \sum_{x,y\in \X} w_{xy}^2(u(x) - u(y))^2 \ \ \mbox{ subject to } u(x) = g(x) \mbox{ for all } x \in \O.
\end{equation}

Laplacian regularization was first proposed for learning tasks in~\cite{zhu2003semi}, and has been extensively used since (see \cite{zhou2005learning,zhou2004learning,zhou2004ranking,ando2007learning} and the references therein). It has also been applied to the related problem of manifold ranking~\cite{he2004manifold,he2006generalized,wang2013multi,yang2013saliency,zhou2011iterated,xu2011efficient}, which has applications in object retrieval, for example.

Although Laplacian regularization is widely used and is successful for certain problems, it has been observed that when the number of unlabeled data points far exceeds the number of labeled points, the solution of \eqref{eq:L2} degenerates into a nearly constant label function $u$, with sharp spikes near each labeled data point $x \in \O$~\cite{nadler2009semi,el2016asymptotic}. See Figure \ref{fig:demoL2} for a simple example of this degeneracy. Thus, in the continuum limit as the number of unlabeled points tends to infinity while the number of labels is finite, the solutions of \eqref{eq:L2} do not continuously attain their boundary data on the set $\O$. In this sense, we say that Laplacian regularization is \emph{ill-posed} in this regime. In general, we will say a learning algorithm is \emph{well-posed} in the limit of infinite unlabeled data and finite labeled data if in the continuum limit the learned functions converge to a continuous function that attains the labels $u=g$ on $\O$ continuously. If the continuum limit does not attain the labels continuously, then we say the algorithm is \emph{ill-posed}.  
\begin{figure}
\centering
\subfigure[$p=2$]{\includegraphics[clip=true,trim = 30 20 30 23, width=0.32\textwidth]{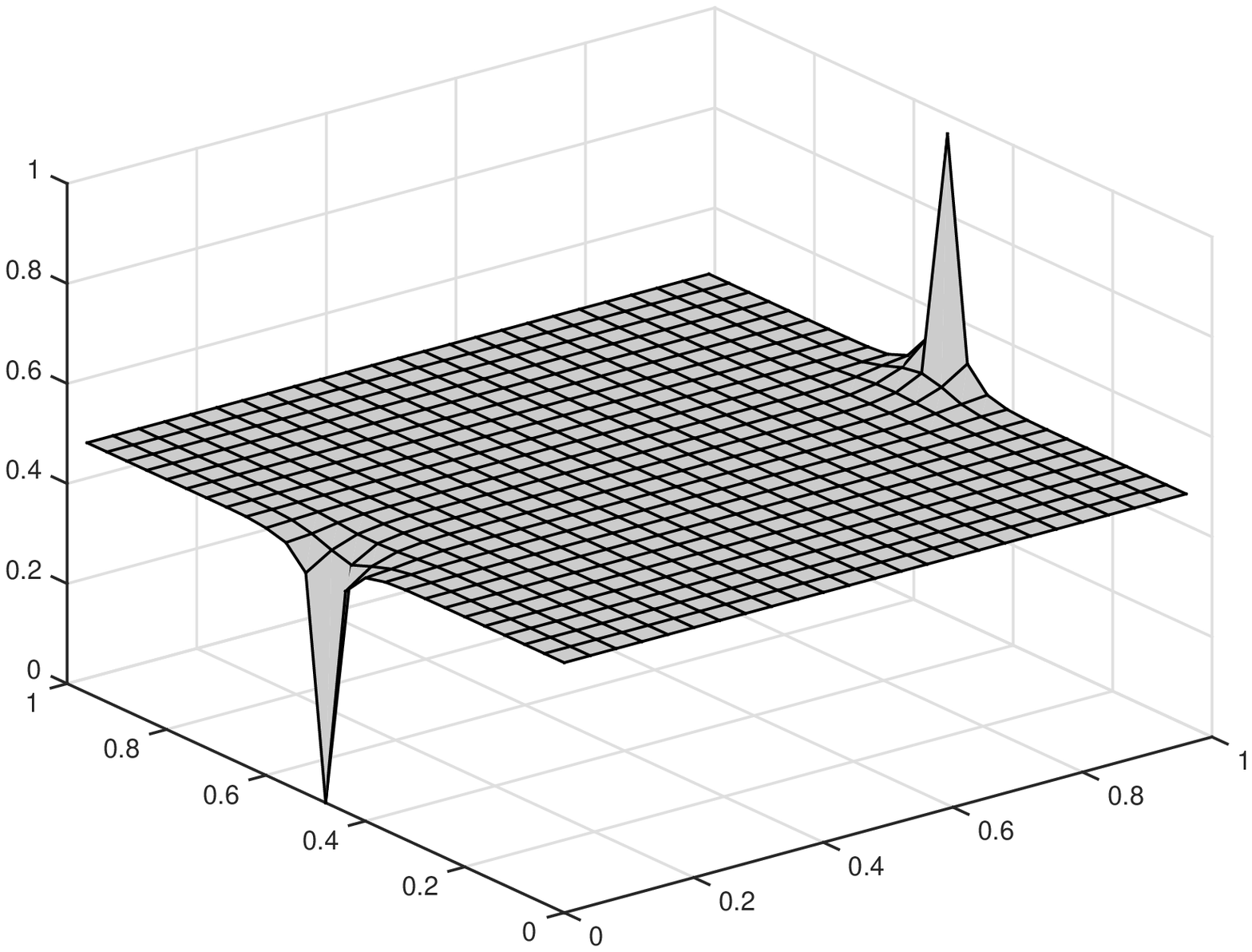}\label{fig:demoL2}}
\subfigure[$p=2.5$]{\includegraphics[clip=true,trim = 30 20 30 23, width=0.32\textwidth]{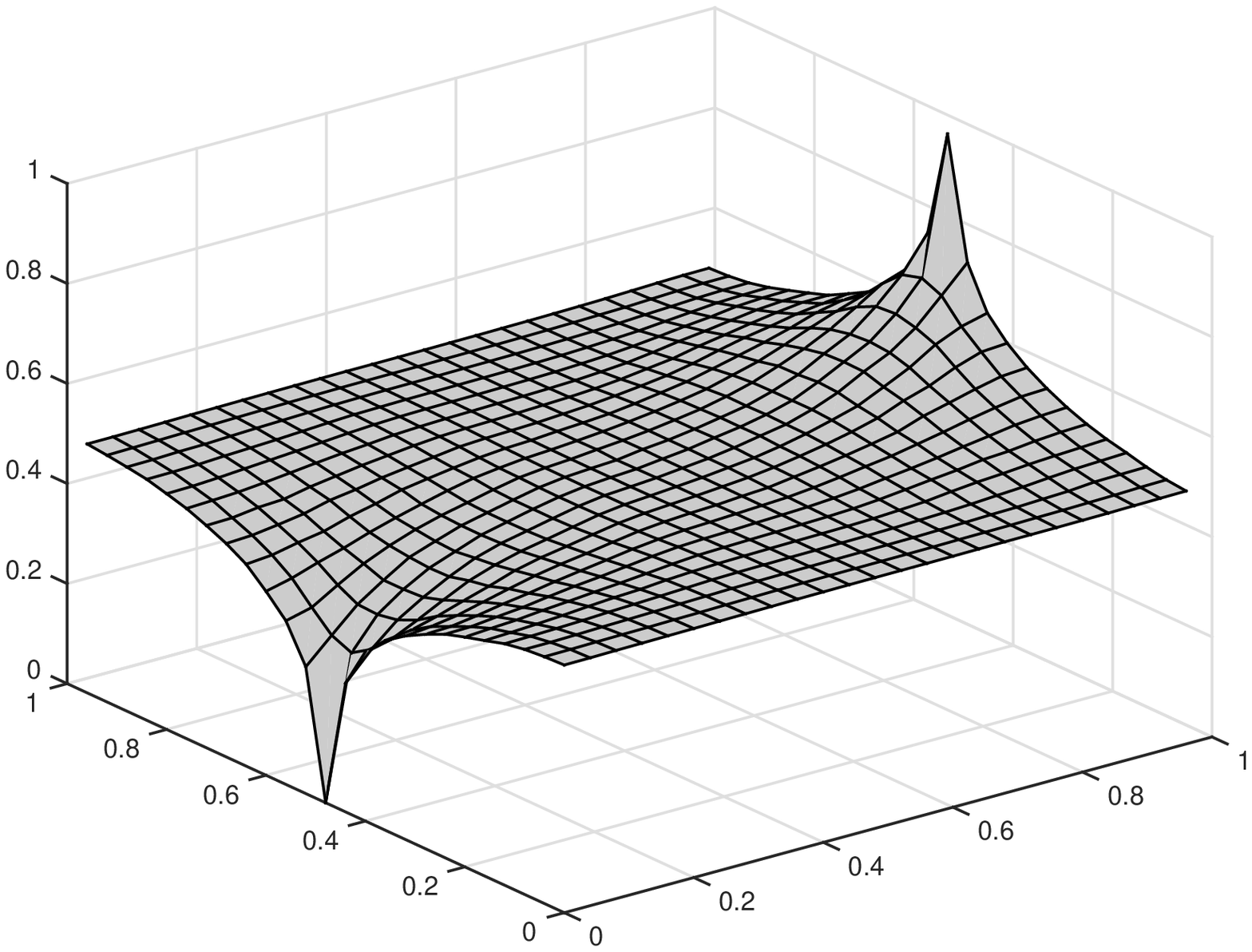}}
\subfigure[$p=\infty$]{\includegraphics[clip=true,trim = 30 20 30 23, width=0.32\textwidth]{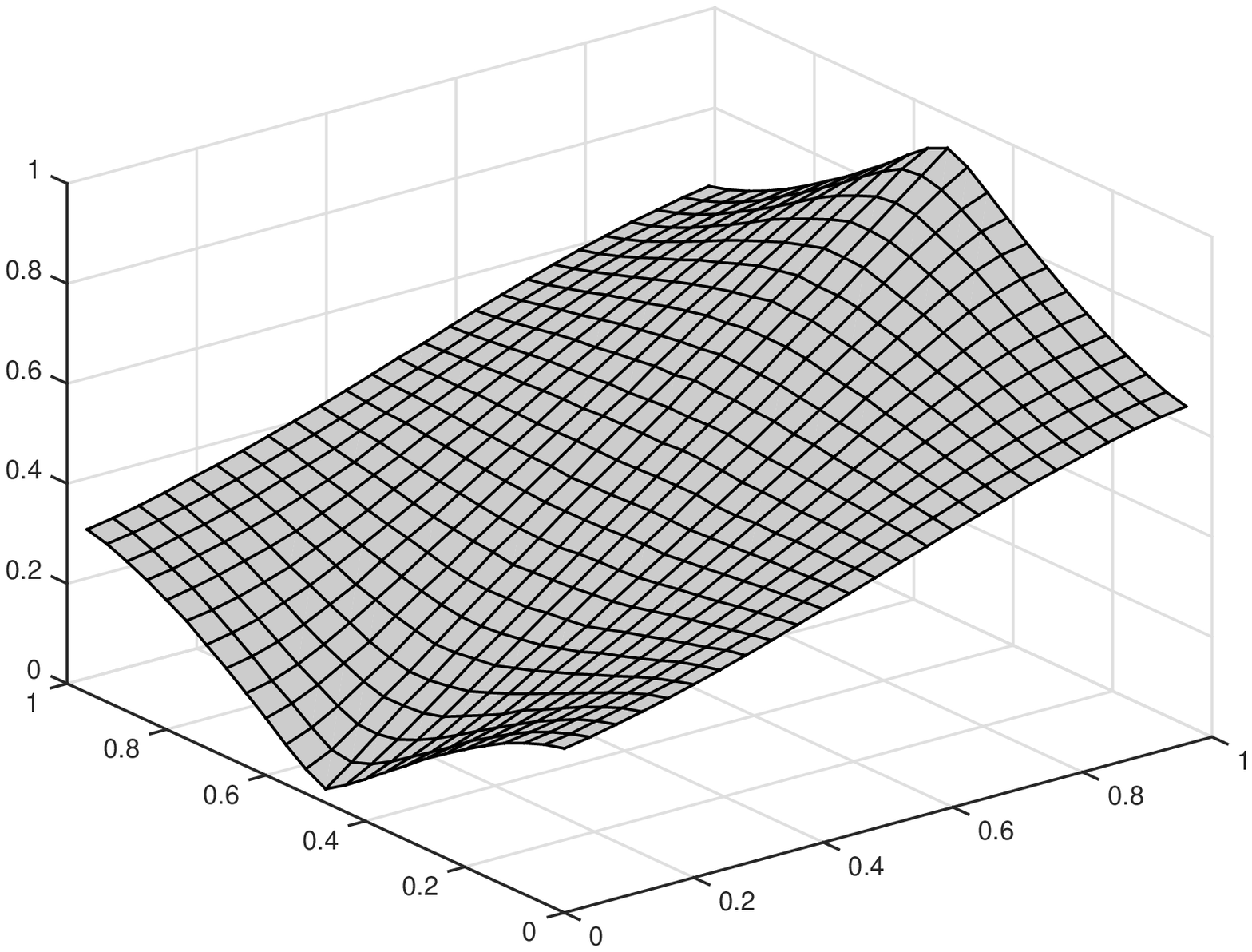}}
\caption{Example of $\ell_p$-Laplacian regularized learning with (a) $p=2$, (b) $p=2.5$ and  (c) $p=\infty$. In each case, the observed (labeled) data is $u(0,0.5)=0$ and $u(1,0.5)=1$, and the unlabeled data is a collection of $n=10^5$ independent and uniformly distributed random variables on $[0,1]^2$. The weights were set to $w_{xy}=1$ if $|x - y| \leq 0.01$ and zero otherwise.}
\label{fig:demo}
\end{figure}

To address the ill-posedness of Laplacian regularization, a class of $\ell_p$-based Laplacian regularization models has recently been proposed \cite{el2016asymptotic}. For $1 \leq p \leq \infty$, $\ell_p$-based Laplacian regularization leads to the optimization problem
\begin{equation}\label{eq:Lp}
\min_{u:\X\to \R} J_p(u) \ \ \mbox{ subject to } u(x) = g(x) \mbox{ for all } x \in \O, 
\end{equation}
where
\begin{equation}\label{eq:Jp}
J_p(u) := \sum_{x,y\in \X}w_{xy}^p|u(x) - u(y)|^p \ \ \mbox{ for } 1 \leq p < \infty,
\end{equation}
and
\begin{equation}\label{eq:Jinfty}
J_\infty(u):= \max_{x,y\in \X} \{w_{xy}|u(x) - u(y)|\}.
\end{equation}

See Figure \ref{fig:demo} for an illustration of $\ell_p$-based Laplacian regularization. As the value of $p$ increases, the smoothness (at least visually) of the minimizer $u$ of \eqref{eq:Lp} increases. The $\ell_p$ models are discussed in~\cite{bridle2013p,alamgir2011phase,el2016asymptotic}, and the $p=\infty$ case--called Lipschitz learning--was studied in~\cite{kyng2015algorithms,luxburg2004distance}. For $p=\infty$, the minimizer is not unique, even when the graph is connected. To see why, suppose we have a minimizer $u$ of $J_\infty(u)$ that satisfies $w_{xy}|u(x)-u(y)| < J_\infty(u)$ for some fixed $x\in \X$ and all $y\in \X$. Then we can modify the value of $u(x)$ by a small amount $\delta>0$ while ensuring that $J_\infty(u)$ is unchanged. This shows that in general there are an infinite number of minimizers of $J_\infty$. To resolve this issue, Kyng et al.~\cite{kyng2015algorithms} suggest to consider the minimizer whose gradient $Du:=(w_{xy}|u(x)-u(y)|)_{x,y\in \X}$ is smallest in the lexicographic ordering on $\R^{|\X|^2}$. This amounts to minimizing the largest value of $Du(x,y)$, and then minimizing the second largest, and then the third largest, and so on. In the continuum setting, this is related to absolutely minimal Lipschitz extensions, which have a long history in analysis~\cite{aronsson2004tour}. 

We note that the $\ell_p$-based Laplacian regularization problem \eqref{eq:Lp} is related in the continuum to the variational problem
\begin{equation}\label{eq:contprob}
\min_{u}\int_\Omega |\nabla u|^p \, dx.
\end{equation}
Minimizers of \eqref{eq:contprob} (for $1 \leq p < \infty$) satisfy the Euler-Lagrange equation 
\begin{equation}\label{eq:pLaplaceEq}
\Delta_p u:= \text{div}(|\nabla u|^{p-2}\nabla u) = 0.
\end{equation}
The operator $\Delta_p$ is called the $p$-Laplacian, and solutions of \eqref{eq:pLaplaceEq} are called $p$-harmonic functions \cite{lindqvist2017notes}.  For $p>2$, the $p$-Laplace equation \eqref{eq:pLaplaceEq} is degenerate, while for $1\leq p <2$ the operator is singular. Solutions of \eqref{eq:pLaplaceEq} can be interpreted in the weak (i.e., distributional) sense, or in the viscosity sense, and solutions are at most $C^{1,\alpha}$ when $p\neq 2$ \cite{lindqvist2017notes}. Notice we can formally expand the divergence in \eqref{eq:pLaplaceEq} to obtain
\[\Delta_p u = |\nabla u|^{p-2}(\Delta u + (p-2)\Delta_\infty u),\]
where the $\infty$-Laplacian $\Delta_\infty$ is given (when $\nabla u\neq 0$) by
\begin{equation}\label{eq:infL}
\Delta_\infty u:= \frac{1}{|\nabla u|^2}\sum_{i,j=1}^d u_{x_ix_j}u_{x_i}u_{x_j}.
\end{equation}
Thus, any $p$-harmonic function satisfies the equation
\begin{equation}\label{eq:gameth}
\Delta u + (p-2)\Delta_\infty u = 0.
\end{equation}
This equation is often called the \emph{game-theoretic} or \emph{homogeneous} $p$-Laplacian\footnote[1]{Due to the identity $\Delta_p u = |\nabla u|^{p-2}(\Delta_1 u + (p-1)\Delta_\infty u)$, it is also common to call the quantity $\Delta_1 u + (p-1)\Delta_\infty u$ the game theoretic $p$-Laplacian \cite{peres2009tug}}, since it arises in two player stochastic tug-of-war games \cite{peres2009tug,lewicka2014game}. At the continuum level, the $p$-Laplace equation \eqref{eq:pLaplaceEq} is equivalent (at least formally) to the game theoretic version \eqref{eq:gameth}, while at the graph level, these two formulations of the \emph{graph} $p$-Laplacian are different.

It was recently proved in \cite{calderLip2017} that in the limit of infinite unlabeled data and finite labeled data, solutions of the Lipschitz learning problem (i.e., \eqref{eq:Lp} with $p=\infty$) converge to an  $\infty$-harmonic function that continuously attains the boundary condition $u=g$ on $\O$. In this sense, we can say that Lipschitz learning is \emph{well-posed} in the limit of finite labeled data and infinite unlabeled data. However, the limiting $\infty$-harmonic function is independent of the distribution of the unlabeled data (i.e., the algorithm ``ignores'' the unlabeled data), and hence Lipschitz learning is fully supervised in this limit \cite{el2016asymptotic,calderLip2017}. 

Along a similar thread, it was shown in  \cite{slepvcev2017analysis} that  $\ell _p$-based Laplacian regularization is well-posed in the limit of finite labeled and infinite unlabeled data when  $p>d$ and $\lim_{n\to \infty}nh_n^p=0$, where $d$ is dimension, $h_n$ is the length scale used to define the weights in the graph (see \eqref{eq:weights}) and $n$ is the number of vertices. This places a restriction on the length scale $h_n \ll (1/n)^{1/p}$ that may be undesirable in practice.  We can see why this restriction is necessary with a basic energy scaling argument. If the $n$ vertices in the graph are randomly sampled from $\R^d$ and vertices are connected on a length scale of order $h_n>0$, then each vertex is connected to on average $O(nh_n^d)$ neighbors, hence the sum defining $J_p(u)$ in \eqref{eq:Jp} has on the order of $n^2h_n^d$ terms. If the function $u$ is a non-constant smooth labeling of the data (i.e., the well-posed regime), then each term in the sum contributes $O(h_n^p)$ to the total, hence the energy scales like $J_p(u) \sim n^2h_n^{d+p}$. On the other hand, if $u$ is constant, except at $L$ labels, then the only nonzero terms in the sum correspond to connections to labeled points, and the size of each term is $O(1)$. Hence, the energy of a constant labeling scales like $J_p(u) \sim Lnh_n^d$. Of course, we want the energy of a constant labeling to be much larger than a non-constant (i.e., $Lnh_n^d \gg n^2 h_n^{d+p}$) to ensure the algorithm is well-posed. This reduces to the requirement that $L \gg nh_n^p$. In the case that $L$ is constant in $n$, then we need $nh_n^p \to 0$ as $n\to \infty$. Let us mention that the authors of \cite{slepvcev2017analysis} showed how to fix this issue by modifying the model \eqref{eq:Lp} so that the boundary condition $u=g$ is required to hold on the dilated set $\O + B(0,2h_n)$. In this case, the length scale restriction $nh_n^p \to 0$ is not required, since we are labeling $L\sim nh_n^d \gg nh_n^p$ data points (recall $p>d$). The disadvantage of the improved model is that it now requires an infinite number of labels, since $nh_n^d\to \infty$ is required for almost sure connectivity of the graph.

In machine learning, it is often assumed that while data may be presented in a very high dimensional space $\R^D$, most types of data lie close to some submanifold of some much smaller dimension. This is called the \emph{manifold assumption}, and the dimension $d\ll D$ of the manifold is often called the \emph{intrinsic dimension} of the data. We note that in the requirement $p>d$ in the well-posedness result from \cite{slepvcev2017analysis}, the value of $d$ should be interpreted as this intrinsic dimension, as opposed to extrinsic dimension $D$.  The intrinsic dimension $d$ can be much larger that $d=2$ in practice. For example, some of the digits in the MNIST dataset\footnote[2]{The MNIST dataset \cite{lecun1998gradient} consists of 70,000 black and white 28x28 pixel images of handwritten digits.} are estimated to have intrinsic dimension between $d=12$ and $d=14$~\cite{hein2005intrinsic,costa2006determining}, and the MNIST dataset has very low complexity compared to modern problems, such as ImageNet.\footnote[4]{The ImageNet dataset~\cite{deng2009imagenet} consists of over $14$ million natural images belonging to over 20,000 categories.} While the Laplacian regularization ($p=2$) problem can be solved efficiently \cite{spielman2004nearly,cohen2014solving}, to our knowledge there are no algorithms in the literature for solving the $\ell_p$-based Laplacian regularization problem \eqref{eq:Lp} efficiently and to scale, especially when $p\gg 2$.  El alaoui et al.~\cite{el2016asymptotic} suggest to use Newton's method, while Kyng et al.~\cite{kyng2015algorithms} suggest convex programming, which requires very high accuracy for large $p$, but neither has been implemented and tested for large scale problems.  To see why it may be difficult to solve \eqref{eq:Lp} when $p$ is large, we recall that the convergence rate for optimization algorithms (such as gradient descent or Newton's method) for solving $\min_u F(u)$ depends on the Lipschitz norm of the gradient
\begin{equation}\label{eq:alpha}
\alpha:= \sup_{u\neq v} \frac{\|\nabla F(u)-\nabla F(v)\|}{\|u-v\|}.
\end{equation}
When $\alpha$ is very large iterative methods are slow to converge~\cite{boyd2004convex}, and we say the problem is poorly conditioned. This is often manifested as an ill-conditioning of the Hessian matrix $\nabla^2 F$. It appears that for $\ell_p$-based learning \eqref{eq:Lp}, $\alpha=\alpha(p)$ grows exponentially with $p$, suggesting the problem is challenging to solve for larger values of $p$.

In this paper we study the game theoretic \emph{graph} $p$-Laplacian as a regularizer for semi-supervised learning on graphs with very few labeled vertices. We show that semi-supervised learning with the game theoretic  $p$-Laplacian is well-posed in the limit of finite labeled data and infinite unlabeled data when $p>d$. 
In particular, we show in this limiting regime that the solutions of the game theoretic graph $p$-Laplacian converge to the unique viscosity solution of a weighted $p$-Laplace equation, and the boundary values $u=g$ on $\O$ are attained continuously in this limit. As in \eqref{eq:gameth}, the game theoretic graph $p$-Laplacian is a linear combination of the  graph $2$-Laplacian and the graph  $\infty$-Laplacian, with the hyperparameter  $p$ appearing as a coefficient. Therefore, the game theoretic  $p$-Laplacian is expected to be well conditioned numerically for all values of  $p$, since the condition number $\alpha$ from \eqref{eq:alpha} can be bounded independently of $p$ (also, see Remark \ref{rem:conditioning} below). Furthermore, the continuum well-posedness of the game-theoretic $p$-Laplacian does not require any upper bound on the length scale $h_n$ (only $h_n \to 0$ is required), which may be more desirable in practice. This work suggests that regularization with the game theoretic $p$-Laplacian may be useful in semi-supervised learning with few labels.   We describe our main results and outline the paper in Section \ref{sec:main} below.

\subsection{Main results}
\label{sec:main}

We now describe our setup and main results. We take periodic boundary conditions and work on the flat Torus $\T^d=\R^d/\Z^d$. Let  $X_n=\{x_1,\dots,x_n\}$ be a sequence of independent and identically distributed random variables on $\T^d$ with probability density $f \in C^2(\T^d)$. We assume that $f>0$ on $\T^d$.  Let $\O \subset \T^d$ be a fixed finite collection of points and set
\begin{equation}\label{eq:graph}
\X_n := X_n\cup \O.
\end{equation}
The points $\X_n$ will form the vertices of our graph. To select the edge weights, let $\Phi:[0,\infty)\to [0,\infty)$ be a $C^2$  function satisfying 
\begin{equation}\label{eq:phi}
\begin{cases}
\Phi (s)\geq 1,&\text{if }s\in (0,1)\\
\Phi (s)=0,&\text{if }s\geq 2.
\end{cases}
\end{equation}
We also assume that there exists  $s_0\in (0,2)$ and $\theta>0$ such that 
\begin{equation}\label{eq:uniquemax}
s\Phi(s) + \theta (s-s_0)^2 \leq s_0 \Phi(s_0) \mbox{ for all } 0\leq s \leq 2.
\end{equation}
 
Select a length scale $h_n>0$ and define
\begin{equation}\label{eq:weights}
w_n(x,y) := \Phi\left( \frac{|x -y|}{h_n} \right), \mbox{ and } d_n(x)=\sum_{y\in \X_n}w_n(x,y).
\end{equation} 
Here, $|x-y|$ denotes the distance on the flat torus $\T^d$.
Let  $\W_n=\left\{ w_n(x,y) \right\}_{x,y\in \X_n}$ and let   $\G_n=(\X_n,\W_n)$ be the graph with vertices  $\X_n$ and edge weights  $\W_n$.  The graph  $2$-Laplacian is given by 
\begin{equation}\label{eq:l2}
L_n^2u(x):=\sum_{y\in \X_n}w_n(x,y)(u(y) -u(x)),
\end{equation}
while the graph  $\infty$-Laplacian is given by  
\begin{equation}\label{eq:li}
L^\infty_nu(x) := \max_{y \in \X_n} w_n(x,y)(u(y) -u(x)) + \min_{y \in \X_n} w_n(x,y) (u(y) -u(x)).
\end{equation}
We define the game theoretic graph $p$-Laplacian~\cite{manfredi2015nonlinear} by
\begin{equation}\label{eq:lp}
L^p_nu(x):=\frac{1}{d_n(x)}L^2_{n}u(x)  +\lambda(p -2)L^\infty_nu(x),
\end{equation}
where
\begin{equation}\label{eq:lam}
\lambda =\frac{C_2}{s_0^2\Phi(s_0)C_1},
\end{equation}
$C_1=\int_{B(0,2)}\Phi (|z|)\, dz$, and $C_2=\frac{1}{2}\int_{B(0,2)}\Phi (|z|)z_1^2\, dz$. The choice of $\lambda$ ensures that for smooth functions $\phi$, $L^p_n\phi=0$ is consistent (as $n\to \infty$ and $h_n\to 0$) with the weighted $p$-Laplacian
\begin{equation}\label{eq:wplap}
\mbox{div}(f^2|\nabla \phi|^{p-2}\nabla \phi) = |\nabla \phi|^{p-2}(\mbox{div}(f^2\nabla \phi) + (p-2)f^2 \Delta_\infty \phi).
\end{equation}

Let  $g:\O\to \R$ be given labels and let  $u_n:\X_n\to \R$ be the solution of the game theoretic  $p$-Laplacian boundary value problem  
\begin{equation}\label{eq:opt}
\left.\begin{array}{rll}
L^p_nu_n&=0&\mbox{in } X_n\\
u_n&=g&\mbox{in } \O.\end{array}\right\}
\end{equation}

Our main result is the following continuum limit. 
\begin{theorem}\label{thm:main}
Let $d < p < \infty$, and suppose that $h_n \to 0$ such that
\begin{equation}\label{eq:hs_finite}
\lim_{n\to \infty} \frac{nh_n^q}{\log(n)} = \infty,
\end{equation}
where $q=\max\{d+4,3d/2\}$.
Then with probability one 
\begin{equation}
u_n \longrightarrow u \ \ \mbox{ uniformly as } \ \ n \to\infty,
\label{eq:conv}
\end{equation}
where $u \in C^{0,\frac{p-d}{p-1}}(\T^d)$ is the unique viscosity solution of the weighted $p$-Laplace equation
\begin{equation}
\left.
\begin{array}{rll}
\div\left(f^2|\nabla u|^{p-2}\nabla u\right) &= 0&\mbox{in } \T^d\setminus \O\\
u&=g&\mbox{on } \O.\\
\end{array}
 \right\}
\label{eq:pLap}
\end{equation}
\end{theorem}
Theorem \ref{thm:main} proves that graph-based semi-supervised learning with the game theoretic $p$-Laplacian is well-posed in the limit of finite labeled data and infinite unlabeled data when $p>d$. We prove in Section \ref{sec:conv} that weak distributional solutions of \eqref{eq:pLap} are equivalent to viscosity solutions, so we can also interpret $u$ as the unique weak solution of \eqref{eq:pLap}. By regularity theory for weak solutions of the weighted $p$-Laplace equation \cite{tolksdorf1984regularity,lieberman1988boundary}, we have $u\in C^{1,\alpha}_{loc}(\T^d\setminus \O)$ and if $f$ is smooth then $u\in C^\infty(B(x,r))$ near any point where $\nabla u(x)\neq 0$.

A key step in proving Theorem \ref{thm:main} is a discrete regularity result that is useful to state independently.
\begin{theorem}\label{thm:holder}
For every  $0<\alpha <(p -d)/(p -1)$  there exists  $C,\delta >0$ such that 
\begin{equation}\label{eq:holder}
 \P\left[ \forall x,y\in \X_n,\,|u_n(x) -u_n(y)|\leq C(|x -y|^\alpha  +h^\alpha _n) \right]
\geq 1 -\exp\left(  -\delta nh_n^{q} +C\log (n) \right),
\end{equation}
 where $q=\max\{d +4,3d/2\}.$   
\end{theorem}
Theorem \ref{thm:holder} quantifies the way in which the game theoretic $p$-Laplacian regularizes the learning algorithm; it asks that $u_n$ is approximately H\"older continuous.

Several remarks are in order.
\begin{remark}
Notice that \eqref{eq:pLap} is the Euler-Lagrange equation for the variational problem
\[\min_{u:\T^d \to \R} \int_{\T^d} f^2|\nabla u|^p \, dx\]
subject to $u=g$ on $\O$. This variational problem is also the continuum limit of $\ell_p$-based Laplacian regularization~\cite{el2016asymptotic}. Hence, $\ell_p$-based Laplacian regularization and regularization via the game theoretic $p$-Laplacian, while different models at the graph level, are identical in the continuum.
\end{remark}
\begin{remark}
We expect that when $p\leq d$, $u_n$ converges pointwise to a constant $c\in \R$. It would be interesting to prove this and determine the constant $c$.
\end{remark}
\begin{remark}\label{rem:Lp}
The proof of Theorem \ref{thm:main} uses techniques from the theory of viscosity solutions of partial differential equations~\cite{crandall1992user}. The same ideas can be used to prove a similar result for the graph  $p$-Laplacian boundary value problem 
\begin{equation}\label{eq:pp}
 \left.\begin{array}{rll}
\sum_{y\in \X_n}w_n(x,y)|u_n(x) -u_n(y)|^{p -2}(u_n(x) -u_{n}(y))&=0&\mbox{if }x\in  X_n\\
u_n(x)&=g(x)&\mbox{if }x\in  \O,\end{array}\right\}
\end{equation}
which corresponds to  $ \ell _p$-based Laplacian regularization. The main difference in the proof is the consistency result, which we include for \eqref{eq:pp} in the appendix for completeness. This problem was studied very recently with variational (as opposed to PDE) techniques by Slep\v cev and Thorpe \cite{slepvcev2017analysis}.  In this setting, there is an additional requirement that $\lim_{n\to \infty} nh_n^p = 0$, which places an upper bound on the length scale $h_n \ll 1/n^{1/p}$ that is not present in the game theoretic model.
\end{remark}
\begin{remark}\label{rem:conditioning}
Theorem \ref{thm:main} suggests that the game theoretic  $p$-Laplacian can be useful for regularization in graph-based semi-supervised learning with small amounts of labeled data.  We note that the constituent terms in the game theoretic $p$-Laplacian---the graph $2$-Laplacian and $\infty$-Laplacian---are well-understood and efficient algorithms exist for solving both \cite{spielman2004nearly,kyng2015algorithms}. Furthermore, on uniform grids (i.e., the PDE-version), there are efficient algorithms for the game theoretic $p$-Laplacian, such as the semi-implicit scheme of Oberman \cite{oberman2013finite}. This suggests the game theoretic $p$-Laplacian on a graph can be solved efficiently and to scale, making it a good choice for semi-supervised learning with few labels.
\end{remark}
\begin{remark}
If instead of working on the Torus  $\T^d$, we take our unlabeled points to be sampled from a domain  $X_n\subset \Omega\subset \R^d$, then we expect  Theorem \ref{thm:main} to hold  under similar hypotheses, with the additional boundary condition
\begin{equation}\label{eq:neumann}
\frac{\partial u}{\partial \nu}=0\quad \mbox{  on }\partial \Omega.
\end{equation}
Let us outline briefly how we expect the proof to change. The Neumann condition \eqref{eq:neumann} will make an appearance in the consistency result for the graph Laplacian (Theorem \ref{thm:consistency2}), which near the boundary $\partial\Omega$ will involve, for the $2$-Laplacian, the integral
\begin{equation}\label{eq:conint}
L_n^2u(x) \approx n\int_{B(x,2h_n)\cap \Omega}\Phi\left( h_n^{-1}|x-y| \right)(u(y) - u(x)) \, dy.
\end{equation}
Taylor expanding $u$ we obtain an expression of the form
\[L_n^2u(x)\approx Cnh_n^{d+1}\left( \frac{\text{dist}(x,\partial\Omega)}{h} - 1 \right)\frac{\partial u}{\partial \nu} + O(nh^{d+2}).\]
A similar statement holds for the $\infty$-Laplacian. The remainder of the proof will be the same; the final difference is that we will require uniqueness of viscosity solutions to the weighted $p$-Laplace equation with Neumann boundary conditions. We refer the reader to the user's guide \cite[Theorem 7.5]{crandall1992user} for a comparison principle for the generalized Neumann problem. Uniqueness of viscosity solutions will require some boundary regularity, and we expect $\partial \Omega\in C^1$ to be sufficient.
\end{remark}

This paper is organized as follows. In Section \ref{sec:maxprinc} we recall the maximum principle on a graph and prove that \eqref{eq:opt} is well-posed. In Section \ref{sec:consistency} we prove consistency for the game theoretic $p$-Laplacian for smooth functions. In Section \ref{sec:holder} we prove a discrete regularity result (Theorem \ref{thm:holder}) for the game theoretic $p$-Laplacian. Finally, in Section \ref{sec:conv} we prove our main consistency result, Theorem \ref{thm:main}.

\section{The maximum principle}
\label{sec:maxprinc}

Our main tool is the maximum principle. We recall here the maximum principle on a graph, which was proved in a similar setting in \cite{manfredi2015nonlinear}. Our setting is slightly more general, so we include a proof here for completeness. 

We first introduce some notation. We say that  $y$ is adjacent to  $x$ whenever  $w_n(x,y)>0$.  We say that the graph  $\G_n=(\X_n,\W_n)$  is connected to $\O\subset \X_n$ if for every  $x\in \X_n\setminus \O$  there exists $y\in \O$  and a path from  $x$ to $y$ consisting of adjacent vertices. 
\begin{theorem}[Maximum principle]\label{thm:maxprinc}
Assume the graph $\G_n$ is connected to $\O$. Let  $p\geq 2$ and suppose $u,v:\X_n \to \R$ satisfy
\begin{equation}\label{eq:supersub}
L^p_nu(x)\geq L^p_nv(x)\quad \mbox{  for all }x\in X_n.
\end{equation}
Then 
\begin{equation}\label{eq:maxprinc}
\max_{\X_n}(u - v) = \max_{\O} (u - v).
\end{equation}
In particular, if $u\leq v$ on $\O$, then $u\leq v$ on $\X_n$.
\end{theorem}
\begin{proof}
 Let $x\in X_n$ be a point at which $u-v$ attains its maximum. If $x \in \O$, then we are done, so suppose that $x \in X_n$. Since  $u -v$ attains its maximum at  $x$ we have 
 \[w_n(x,y)(u(y) -u(x))\leq w_n(x,y)(v(y) -v(x))\quad \mbox{ for all }y\in \X_n.\] 
If any of the inequalities above were strict, we would have  $L^p_nu(x)<L^p_nv(x)$ which contradicts the assumption given in \eqref{eq:supersub}. Therefore 
\[u(y) -v(y)=u(x) -v(x)\quad \mbox{ whenever }w_n(x,y)>0.\] 
That is,  $u -v$  also attains its maximum at every vertex that is adjacent to  $x$. Since the graph is connected to  $\O$, we can find a path from  $x$ to  $\O$ consisting of adjacent vertices. Therefore  $u -v$ attains its maximum somewhere on  $\O$, which completes the proof.      
\end{proof}
\begin{remark}
Notice the exact form of $L^p_n$ is not used directly in the proof of Theorem \ref{thm:maxprinc}. All that is required is that $L^p_n u(x)$ is a strictly increasing function of $w_n(x,y)(u(y)-u(x))$ for every $y\in \X_n$ with $y\neq x$. Finite difference schemes with this property are called \emph{monotone} in the PDE literature~\cite{barles1991convergence}.
\end{remark}
We now  establish existence and uniqueness of a solution to \eqref{eq:opt}.
\begin{theorem}\label{thm:exist}
Assume $ \G_n$ is connected to $\O$ and let  $p\geq 2$. For any $g:\O \to \R$, there exists a unique solution $u_n:\X_n\to \R$ of \eqref{eq:opt}. Furthermore, $u$ satisfies the \emph{a priori} estimate
\begin{equation}\label{eq:apriori}
\min_{\O}g\leq u_n\leq \max_{\O}g.
\end{equation}
\end{theorem}
\begin{proof}
 We use the Perron method to prove existence. Define
\[\F := \left\{v:\X_n\to \R \, : \, L^p_n v\geq 0 \mbox{ in } X_n \mbox{ and } \min_\O g \leq v \leq g \mbox{ in } \O\right\}\]
and for  $x\in \X_n$  set 
\[u_n(x):=\sup_{v \in \F} v(x).\]
Note that $v\equiv\min_\O g$ belongs to $\F$, so $\F$ is nonempty. Furthermore, $w\equiv \max_\O g$ satisfies $L^p_nw\equiv 0$, and so by Theorem \ref{thm:maxprinc}, $v \leq w$ for all $v \in \F$. It follows that  $u_n$ satisfies \eqref{eq:apriori}.   

We now show that $L^p_n u_n \geq 0$ in $X_n$. Fix $x \in X_n$ and let $v_k \in \F$ be a sequence of functions such that $v_k(x) \to u_n(x)$ as $k\to \infty$.  Since $\min_\O g \leq v_k \leq \max_\O g$ we can pass a subsequence, if necessary, so that $v_k \to v$ on $\X_n$ for some $v:\X_n\to \R$. By continuity $L^p_n v \geq 0$ on $X_n$ and $\min_\O g \leq v \leq g$ in $\O$. Therefore $v \in \F$, and so $u_n \geq v$ on $\X_n$ and $u_n(x)=v(x)$. It follows that   $L^p_n u_n(x) \geq L^p_n v(x) \geq 0$.

We now show that $L_n^p u_n \leq 0$ in $X_n$. Fix $x \in X_n$ and assume to the contrary that $L^p_n u_n(x)>  0$.
Define
\[w(y):=\begin{cases}
1,&\text{if } y=x\\
0,&\text{otherwise.}\end{cases}\]
For $\eps>0$ define $v_\eps:=u_n + \eps w$. By continuity, we can choose $\eps>0$ sufficiently small so that $L^p_nv_\eps (x)\geq 0$. By definition we have  $L^p_nv_\eps (y)\geq L^p_nu_n(y)\geq 0$ for all  $y\neq  x$.  Therefore $L^p_nv_\eps \geq 0$ in $X_n$ for sufficiently small $\eps>0$. Hence $v_\eps \in \F$ and so $u_n(x) \geq v_\eps(x) = u_n(x) + \eps$, which is a contradiction.
 
 Finally we show that  $u_n=g$ on  $\O$.   By definition, we have that  $u_n\leq g$ on  $\O$. Assume to the contrary that  $u_n(x)<g(x)$ for some  $x\in \O$. Then as above we  can set $v_\eps =u_n +\eps w$, and we find that  $v_\eps \in \F $ for  $\eps >0$ sufficiently small, which is a contradiction. Uniqueness follows directly from Theorem \ref{thm:maxprinc}. 
\end{proof}

\section{Consistency}
\label{sec:consistency}

We now prove consistency for the game theoretic $p$-Laplacian applied to smooth test functions. This is split into two steps: In Section \ref{sec:2lap} we prove a consistency result for the graph $2$-Laplacian and in Section \ref{sec:ilap} we prove a consistency result for the graph $\infty$-Laplacian.  We note there are numerous consistency results in the literature for the graph $2$-Laplacian \cite{hein2007graph,hein2006uniform,belkin2005towards,gine2006empirical,hein2005graphs,singer2006graph,ting2010analysis}, but we require a slightly different notion of consistency to use the viscosity solution framework and prove discrete regularity.

\subsection{Concentration of measure}
\label{sec:concentration}

Graph Laplacians are sums of \emph{i.i.d.}~random variables, and to prove consistency we need to control the fluctuations of the graph Laplacian about its mean. Our main tool for proving consistency is a standard concentration of measure result referred to as Bernstein's inequality \cite{boucheron2013concentration}, which we recall now for the reader's convenience. For $Y_1,\dots,Y_n$ \emph{i.i.d.}~with variance $\sigma^2 = \E((Y_i-\E[Y_i])^2)$, if $|Y_i|\leq M$ almost surely for all $i$ then Bernstein's inequality states that for any $t>0$
\begin{equation}\label{eq:bernstein}
\P\left( \left| \sum_{i=1}^n Y_i - \E(Y_i) \right|> t\right)\leq 2\exp\left( -\frac{t^2}{2n\sigma^2 + 4Mt/3} \right).
\end{equation}
Since the graph Laplacian depends only on nearby nodes in the graph, the variance $\sigma^2$ is much smaller than the absolute bound $M$, which makes the Bernstein inequality far tighter (for small $t$) than other concentration inequalities, such as the Hoeffding inequality~\cite{hoeffding1963probability}, which only depends on the size of $M$. 

The following lemma converts the Bernstein inequality into a result that is directly useful for graph Laplacians.
\begin{lemma}\label{lem:ch}
Let  $Y_1,Y_2,Y_3,\dots,Y_n$  be  a sequence of \emph{i.i.d}~random variables on  $\R^d$ with Lebesgue density $f:\R^d\to \R$, let $\psi:\R^d \to \R$ be bounded and Borel measurable with compact support in a bounded open set $\Omega\subset \R^d$, and define
\[ Y = \sum_{i=1}^n \psi(Y_i).\]
Then for any $0 \leq \lambda \leq 1$
\begin{equation}\label{eq:con}
\P\left( \left| Y - \E(Y) \right|> \|f\|_\infty\|\psi\|_\infty n|\Omega|\lambda\right)\leq 2\exp\left( -\frac{1}{4}\|f\|_\infty n|\Omega| \lambda^2 \right),
\end{equation}
where $\|\psi\|_\infty=\|\psi\|_{L^\infty(\Omega)}$, and $|\Omega|$ denotes the Lebesgue measure of $\Omega$.
\end{lemma}
\begin{proof}
The proof is a direct application of the Bernstein inequality \eqref{eq:bernstein}. We compute $M \leq \|\psi\|_\infty$ and
\[\sigma^2 \leq \E(\psi(Y_i)^2) = \int_\Omega \psi(y)^2 f(y) \, dy\leq |\Omega|\|\psi\|_\infty^2 \|f\|_\infty.\] 
Therefore, Bernstein's inequality yields
\begin{equation}\label{eq:bernstein_app}
\P\left( \left| Y - \E(Y) \right|> t\right)\leq 2\exp\left( -\frac{t^2}{2n|\Omega|\|\psi\|_\infty^2\|f\|_\infty + 4\|\psi\|_\infty t/3} \right),
\end{equation}
for any $t>0$. Setting $t=\|f\|_\infty\|\psi\|_\infty n|\Omega|  \lambda$ for $\lambda>0$ we have
\begin{equation}\label{eq:bernstein_app2}
\P\left( \left| Y - \E(Y) \right|> \|f\|_\infty\|\psi\|_\infty n|\Omega|\lambda\right)\leq 2\exp\left( -\frac{\|f\|_\infty n|\Omega| \lambda^2}{2 + 4\lambda/3} \right).
\end{equation}
Restricting $\lambda\leq 1$ completes the proof.
\end{proof}
\begin{remark}
If $\Omega=B(x,2h)$ in Lemma \ref{lem:ch}, then under the same assumptions as in Lemma \ref{lem:ch} we have
\begin{equation}\label{eq:con2}
\P\left( |Y-\E(Y)|\geq C\|\psi\|_\infty nh^d\lambda\right) \leq 2\exp(-cnh^d\lambda^2),
\end{equation}
for all $0 < \lambda \leq 1$, where $C,c>0$ are constants depending only on $\|f\|_\infty$ and $d$. 
\label{rem:ch}
\end{remark}

\subsection{Consistency of  the $2$-Laplacian}
\label{sec:2lap}

We now prove consistency for the graph $2$-Laplacian.
\begin{theorem}\label{thm:consistency2}
Let $\delta>0$ and let   $\Omega _{\delta ,n}$  denote the event that 
\begin{equation}\label{eq:consistencyd}
\left|\frac{1}{nh_n^d}d_n(x) -C_1f(x) \right|\leq C\delta h_n,
\end{equation}
and
\begin{equation}\label{eq:consistency2}
\left|\frac{1}{nh_n^{d +2}}L^2_n\varphi (x)-C_2f^{ -1}\mbox{div}\left( f^2\nabla  \varphi  \right)  \right|\leq C \|\varphi \|_{C^3(B(x,2h_n))}\delta,
\end{equation}
hold for all  $x\in X_n$ with $\mbox{dist}(x,\O)> 2h_n$ and  $\varphi \in C^3(\R^d)$.  Then there exists  $C,c>0$  such that for  $h_n\leq \delta \leq h_n^{-1}$ we have     
\begin{equation}\label{eq:probability}
\P[\Omega  _{\delta ,n}]\geq 1 -C\exp\left(  -c\delta ^2nh_n^{d +2} + \log(n)\right).
\end{equation}
\end{theorem}
\begin{remark}
We mention that $d_n(x)$ is a kernel density estimator for the density $f(x)$, and the error estimate \eqref{eq:consistencyd} is standard in the density estimation literature \cite{silverman2018density}.
\end{remark}
\begin{proof}
By conditioning on the location of $x\in X_n$, we can assume without loss of generality that $x\in \T^d$ is a fixed (non-random) point. Write $\beta = \|\phi\|_{C^3(B(x,2h_n))}$. Let $\varphi \in C^3(\R^d)$ and set  $p=D\varphi (x)$ and  $A=D^2\varphi (x)$. Note that 
\begin{align}\label{eq:ts}
 L_n^2\varphi (x)=\sum_{i=1}^dp_i\sum_{y\in X_n}w_n(x,y)(y_i -x_i) &+\frac{1}{2}\sum_{i,j=1}^da_{ij}\sum_{y\in X_n}w_n(x,y)(y_i -x_i)(y_j -x_j) \nonumber\\
&\hspace{1.5in}+O\left( h^3\beta d_n(x) \right).
\end{align}
Let $\delta>0$. By Lemma \ref{lem:ch} and Remark \ref{rem:ch}, each of
\[\left|d_n(x) -n\int_{B(x,2h_n)}\Phi \left( \frac{|x -y|}{h_n} \right)f(y)\, dy  \right|\geq C\delta nh_n^{d +1},\] 
\[\left|\sum_{y\in X_n}w_n(x,y)(y_i -x_i) -n\int_{B(x,2h_n)}\Phi \left( \frac{|x -y|}{h_n} \right)(y_i -x_i)f(y)\, dy  \right|\geq C\delta nh_n^{d +2},\] 
and
\begin{align*}
\left|\sum_{y\in X_n}w_n(x,y)(y_i -x_i)(y_j-x_j) -n\int_{B(x,2h_n)}\Phi \left( \frac{|x -y|}{h_n} \right)(y_i -x_i)(y_j-x_j)f(y)\, dy  \right|&\\
&\hspace{-2cm}\geq C\delta nh_n^{d +3},
\end{align*}
occur with probability at most $2\exp\left(  -c\delta ^2nh_n^{d +2} \right)$ provided $0 < \delta h_n \leq 1$.
Thus, if $h_n \leq \delta\leq h_n^{-1}$ we have 
\begin{equation}\label{eq:app}
\frac{1}{nh_n^{d+2}}L_n^2\varphi (x)=\int_{B(0,2)}\Phi \left( |z| \right)\left( \frac{1}{h_n}p\cdot z +\frac{1}{2}z\cdot A z \right)f(x+zh_n)\, dz  +O\left( \delta \beta \right)
\end{equation}
holds for all  $\varphi \in C^3(\R^d)$ with probability at least $1 -C\exp\left(  -c\delta ^2nh_n^{d +2} \right).$
Notice that 
\begin{align*}
 \frac{1}{h_n}\int_{B(0,2)}\Phi (|z|)(p\cdot z)f(x +zh_n)\, dy&=\nabla  f(x)\cdot \int_{B(0,2)}\Phi (|z|)(p\cdot z)z\, dz  +O(\beta h_n)\\
&=\nabla  f(x)\cdot \sum_{i=1}^dp_i\int_{B(0,2)}\Phi (|z|)z_iz\, dz+O(\beta h_n)\\
&=2C_2\nabla  f(x)\cdot p+O(\beta h_n),
\end{align*}
and
\begin{align*}
\frac{1}{2}\int_{B(0,2)}\Phi \left( |z| \right)(z\cdot A z)f(x+zh_n)\, dz &=\frac{1}{2}f(x)\sum_{i,j=1}^da_{ij}\int_{B(0,2)}\Phi(|z|)z_iz_j\,dz  +O(h_n\beta )\\
&=C_2f(x)\mbox{Trace}(A) +O(h_n\beta ).
\end{align*}
Combining this with \eqref{eq:app} we have that
\begin{equation}\label{eq:app2}
\frac{1}{nh_n^{d +2}}L^2_n\varphi (x)=C_2f^{ -1}\mbox{div}\left( f^2\nabla  \varphi  \right)  +O(\delta \beta ),
\end{equation}
holds with probability at least $1 -C\exp\left(  -c\delta ^2nh_n^{d +2}\right)$. The proof is completed by union bounding over all $x\in X_n$.
\end{proof}

\subsection{Consistency of  the $\infty$-Laplacian}
\label{sec:ilap}

We now prove consistency of the $\infty$-Laplace operator. We define
\begin{equation}\label{eq:T}
H_nu(x) := \max_{y \in \T^d} w_n(x,y)(u(y) -u(x)) + \min_{y \in \T^d} w_n(x,y) (u(y) -u(x)),
\end{equation}
Note that $H_n$ is a non-local operator on the Torus, and does not depend on the realization of the random graph $\G_n$. Recalling the definition of the weights \eqref{eq:weights}, we see that $H_n$ depends only on the choice of length scale $h_n$ in the problem.

We first recall a result from \cite{calderLip2017}.
\begin{lemma}[Lemma 1 in \cite{calderLip2017}]\label{lem:approximation1}
Let  $\varphi \in C^2(\R^d)$.  Then 
\begin{equation}\label{eq:approximation1}
|L^{\infty}_n\varphi (x) -H_n\varphi (x)|\leq C\left( \|\phi\|_{C^1(B(x,2h_n))}+h_n \|\phi\|_{C^2(B(x,2h_n))}\right) r^2_nh_n^{ -1},
\end{equation}
where 
\begin{equation}\label{eq:r}
r_n=\sup_{y\in \T^d}\dist(y,\X_n).
\end{equation}
\end{lemma}
Due to Lemma \ref{lem:approximation1}, we only require a consistency result for $H_n$.
\begin{theorem}\label{thm:liconsistency}
For any   $x_0\in \T^d\setminus \O$ and  $\varphi \in C^3(\R^d)$  with  $\nabla  \varphi (x_0)\neq  0$  
\begin{equation}\label{eq:liconsistency}
 \lim_{\substack{n\to \infty \cr x\to x_0}}\frac{1}{h_n^2}H_n\varphi (x)=s_0^2\Phi (s_0)\Delta _\infty\varphi(x_0),
\end{equation}
where
\[\Delta_\infty \varphi = \frac{1}{|\nabla \varphi|^2}\sum_{i,j=1}^d \varphi_{x_ix_j}\varphi_{x_i}\varphi_{x_j}.\]
\end{theorem}
\begin{proof}
As in \cite[Lemma 1]{calderLip2017} we define 
\[B(x) = \max_{0 \leq s \leq 2} \left\{ \frac{1}{h_n}s\Phi(s)|\nabla  \phi(x)| - \frac{1}{2}s^2\Phi(s)\Delta_\infty \phi(x)\right\},\]
and
\begin{equation*}
B'(x) = \max_{0 \leq s \leq 2} \left\{ \frac{1}{h_n}s\Phi(s)|\nabla  \phi(x)| + \frac{1}{2}s^2\Phi(s)\Delta_\infty \phi(x) \right\},
\end{equation*}
and we have
\begin{equation*}
\frac{1}{h_n^2}H_n\varphi (x)= B'(x) -B(x)+ O(\|\phi\|_{C^3(B(x,2h_n))}h_n).
\end{equation*}
Let  $x_n\to x_0$ and for each  $n$  let  $s_n,s'_n>0$ be such that  
\[B(x_n) =\frac{1}{h_n}s_n\Phi(s_n)|\nabla  \phi(x_n)| - \frac{1}{2}s_n^2\Phi(s_n)\Delta_\infty \phi(x_n),\]
and
\[B'(x_n) =\frac{1}{h_n}s'_n\Phi(s'_n)|\nabla  \phi(x_n)| + \frac{1}{2}(s'_n)^2\Phi(s'_n)\Delta_\infty \phi(x_n).\]
Then we have that 
\[\frac{1}{h_n^2}H_n\varphi (x_n)\leq (s_n')^2\Phi (s'_n)\Delta _\infty\varphi (x_n) + O(\|\phi\|_{C^3(B(x_n,2h_n))}h_n),\]
and
\[\frac{1}{h_n^2}H_n\varphi (x_n)\geq s_n^2\Phi (s_n)\Delta _\infty\varphi (x_n) + O(\|\phi\|_{C^3(B(x_n,2h_n))}h_n).\]
 Since  $s\mapsto s\Phi (s)$ has a unique maximum  $s_0\in (0,2]$,  we find that
 \[s_n,s'_n\to s_0\mbox{  as }n\to \infty.\]  
 It follows that 
 \[\lim_{n\to \infty}\frac{1}{h_n^2}H_n\varphi (x_n)=s_0^2\Phi (s_0)\Delta _\infty\varphi(x_0).\qedhere \]
\end{proof}

\section{Regularity}
\label{sec:holder}

Here, we prove a discrete H\"older regularity result (Theorem \ref{thm:holder}) for the solutions $u_n$ of \eqref{eq:opt}.  The proof is based on a well-known trick for establishing H\"older regularity for solutions of the unweighted $p$-Laplace equation
\begin{equation}\label{eq:pppp}
\mbox{div}(|\nabla u|^{p-2}\nabla u) = 0
\end{equation}
via the maximum principle. The trick is to notice that $v(x) = |x-y|^{\frac{p-d}{p-1}}$ is a solution of \eqref{eq:pppp} away from $x=y$. When $p>d$, $v$ is continuous and if we choose a large enough constant $C$ so that
\begin{equation}\label{eq:hpp}
u(y) - Cv(x) \leq u(x) \leq u(y) + Cv(x)
\end{equation}
for all boundary points $x$, then the maximum principle can be invoked to establish that \eqref{eq:hpp} holds for all $x$, i.e.,
\[|u(y)-u(x)| \leq C|x-y|^{\frac{p-d}{p-1}}.\]
The function $v$ is called a \emph{barrier} in the PDE literature.

Adapting this to the graph setting is somewhat technical, since $v(x)$ is not an exact solution of the game theoretic $p$-Laplacian, due to random fluctuations and errors in the consistency results. The argument can be rescued by using the barrier $w(x)=|x-y|^\alpha$ for any $\alpha<(p-d)/(p-1)$. This function is a strict supersolution of \eqref{eq:pppp}, which allows some room to account for the difference between the graph and continuum $p$-Laplacians. It is possible to show (see Lemma \ref{lem:super}) that $w$ is a supersolution of the game theoretic graph $p$-Laplacian (i.e., $L^p_nw(x)\leq 0$) for $|x-y|\geq ch_n$ for some $c>0$ with high probability. The errors in the consistency results blow up as we approach the singularity at $x=y$, so $w$ is not a global supersolution. In Lemma \ref{lem:finesuper}, we show how to modify $w$ near $x=y$ to ensure the supersolution property holds globally. The modification relies extensively on the presence of the graph $\infty$-Laplacian term.

For notational simplicity, we set
\[\Delta_2 \phi = f^{-1}\div(f^2\D \phi).\]

We first require some elementary propositions.
\begin{proposition}\label{prop:supersolution}
For  $p>d$ and $\alpha\in (0,1)$  the function  $v(x)=|x|^\alpha$ satisfies 
\begin{equation}\label{eq:supersolution}
\Delta _2 v(x)\leq f(x)\alpha|x|^{\alpha -2}\left( d + \alpha-2 +C|x| \right).
\end{equation}
\end{proposition}
\begin{proof}
Notice that 
\[\Delta_2 v = f\Delta v + 2\nabla f\cdot \nabla v.\]
The proof is completed by computing 
\[\Delta v=\alpha|x|^{\alpha -2}\left( d + \alpha-2 \right)\mbox{ and }\nabla  v=\alpha|x|^{\alpha-2}x.\qedhere\]
\end{proof}
\begin{proposition}\label{prop:supersolution2}
Let   $p>d$ and $\alpha\in (0,1)$. For every  $\delta >0$ there exists  $C\geq 2$ such that     the function  $v(x)=|x|^\alpha$ satisfies 
\begin{equation}\label{eq:supersolution2}
\frac{1}{h_n^2}H_nv(x)\leq (1 -\delta ) s^2_0\Phi (s_0)\alpha(\alpha-1)|x|^{\alpha-2}
\end{equation}
for all  $|x|\geq Ch_n$.  
\end{proposition}
\begin{proof}
Note that 
\begin{align*}
 \hspace{1cm} H_nv(x)&=\max_{z\in B(0,2)}\Phi (|z|)\left( |x +h_nz|^\alpha  -|x|^\alpha  \right) +\min_{z\in B(0,2)}\Phi (|z|)\left( |x +h_nz|^\alpha  -|x|^\alpha  \right)\\
&= \max_{0\leq s\leq 2}\Phi (s)\left( (|x| +sh_n)^\alpha  -|x|^\alpha  \right) +\min_{0\leq s\leq 2}\Phi (s)\left( (|x| -sh_n)^\alpha  -|x|^\alpha  \right).
\end{align*}
Let  $s_n\geq 0$  such that
\[A_n:=\Phi (s_n)\left( (|x| +s_nh_n)^\alpha  -|x|^\alpha  \right)= \max_{0\leq s\leq 2}\Phi (s)\left( (|x| +sh_n)^\alpha  -|x|^\alpha  \right).\] 
Then we have that 
\[H_nv(x)\leq \Phi (s_n)\left( (|x| +s_nh_n)^\alpha  +(|x| -s_nh_n)^\alpha  -2|x|^\alpha  \right).\]
By Taylor expanding the right-hand side we have 
\begin{equation}\label{eq:up}
H_nv(x)\leq s_n^2\Phi (s_n)\alpha (\alpha  -1)|x|^{\alpha  -2}h_n^2.
\end{equation}

Note that 
\begin{align*}
A_n&=\Phi (s_n)\left( (|x| +s_nh_n)^\alpha  -|x|^\alpha  \right)\\
&\leq \Phi (s_n)\left( |x|^\alpha  +\alpha s_nh_n|x|^{\alpha  -1} -|x|^\alpha  \right)\\
&=s_n\Phi (s_n)\alpha |x|^{\alpha  -1}h_n.
\end{align*}
On the other hand, we have 
\begin{align*}
A_n&\geq \Phi (s_0)\left( (|x| +s_0h_n)^\alpha  -|x|^\alpha  \right)\\
&\geq \Phi (s_0)\left( |x|^\alpha  +\alpha s_0h_n|x|^{\alpha  -1} +\frac{\alpha (\alpha  -1)}{2}s_0^2h_n^2|x|^{\alpha  -2} -|x|^\alpha  \right)\\
&=\frac{1}{2}s_0\Phi (s_0)\alpha |x|^{\alpha  -1}h_n\left(2 +(\alpha  -1)s_0h_n|x|^{ -1}  \right).
\end{align*}
Combining these two inequalities we find that 
\begin{equation}\label{eq:lb1}
s_n\Phi (s_n)\geq \left(1 - \tfrac{1}{2}(1 -\alpha )s_0h_n|x|^{ -1} \right)s_0\Phi (s_0).
\end{equation}
By \eqref{eq:uniquemax} we have 
\[\theta (s_n -s_0)^2\leq s_0\Phi (s_0) -s_n\Phi (s_n)\leq \frac{1}{2}(1 -\alpha )s_0^2\Phi (s_0)h_n|x|^{ -1}.\]
Therefore 
\[s_n\geq \left( 1 -\theta ^{ -1/2}(1 -\alpha )^{1/2}h_n^{1/2}|x|^{ -1/2}\right)s_0.\]
Combining this with \eqref{eq:lb1}, we see that for every $\delta>0$ there exists $C>0$ such that
\[s_n^2\Phi(s_n) \geq (1-\delta)s_0^2\Phi(s_0)\]
for $|x|\geq Ch_n$.
The proof is completed by inserting this into \eqref{eq:up} and noting that $\alpha-1<0$.
\end{proof}

We now establish that our barrier function $|x-y|^\alpha$ is a supersolution away from $x=y$.
\begin{lemma}\label{lem:super}
Let  $p>d$,  $\sigma \in [0,1]$,  $\alpha <(p -d)/(p -1)$,  and for $y\in \T^d$, and set $v_y(x)=|x-y|^{\alpha }$.
 Then there exists  $C,c,\delta,r>0$, depending only on $d$, $p$, and $\alpha $, so that 
\begin{align*}
 \P\left[\forall y\in \X_n,x\in X_n\cap B(y,r)\setminus B^0(y,ch_n^\sigma),\ L^p_nv_y(x)\leq 0\right]\geq 1-\exp\left(  -\delta nh_n^{q} +C\log(n) \right),
\end{align*}
where $q = \max\{d+2+2\sigma,3d/2\}$. 
Furthermore, if  $\sigma <1$ then the Lemma holds for any $c>0$, provided $h_n$ is sufficiently small, and $C,\delta$, and $r$ now additionally depend on $c$.
\end{lemma}
\begin{proof}
Contrary to Theorem \ref{thm:consistency2} we may have $\mbox{dist}(x,\O) \leq 2h_n$, and must account for this. Notice that
\[L^2_nv_y(x) = \sum_{z\in X_n} w_n(x,z)(v_y(z)-v_y(x)) + \sum_{z\in \O} w_n(x,z)(v_y(z)-v_y(x)),\]
and
\[d_n(x) = \sum_{z\in X_n} w_n(x,z) + \sum_{z\in \O} w_n(x,z).\]
The argument from Theorem \ref{thm:consistency2} applies to the summations over $X_n$ in the equations above, and the summations over $\O$ are bounded by a constant.  Thus, assuming $nh_n^{d+2}\geq 1$ and $|x-y|\geq \max\{ h_n^\sigma,3h_n\}$, Theorem \ref{thm:consistency2} and Proposition \ref{prop:supersolution} yield
\begin{align}\label{eq:L2error}
\frac{L^2_nv_y(x)}{d_n(x)h_n^2}&\leq \frac{C_2\Delta _2v_y(x) +C(|x-y|^{\alpha  -3}\delta h_n^\sigma + n^{-1}h_n^{-d-2}) }{C_1f(x) -C(\delta h_n^{1 +\sigma }+n^{-1}h_n^{-d})}\nonumber\\
&\leq \frac{C_2\alpha f(x) |x-y|^{\alpha  -2}(d +\alpha  -2 +C|x-y| +C\delta)}{C_1f(x) -C\delta h_n^{1 +\sigma }}\nonumber\\
&\leq \frac{C_2}{C_1}\alpha |x-y|^{\alpha  -2}(d +\alpha  -2 +C|x-y| +C\delta).
\end{align}
holds for all $y\in \X_n$ and  $x\in X_n$ with $|x-y|\geq \max\{ h_n^\sigma,3h_n\}$ with probability at least
\begin{equation}\label{eq:probability2}
1-C\exp\left(  -\delta ^2nh_n^{d +2 +2\sigma } +\log(n) \right).
\end{equation}
For the rest of the proof we assume that \eqref{eq:L2error} holds.

Let  $\delta '>0$. By Proposition \ref{prop:supersolution2} there exists  $C_{\delta '}$ such that 
\[\frac{1}{h_n^2}H_nv_y(x)\leq (1 -\delta ')s_0^2\Phi (s_0)\alpha (\alpha  -1)|x-y|^{\alpha  -2}\]
for all  $|x-y|\geq C_{\delta '}h_n$.   By Lemma \ref{lem:approximation1} we have 
\[L^\infty_nv_y(x)\leq H_nv_y(x) +C(|x-y|^{\alpha  -1} +h_n|x-y|^{\alpha  -2})r_n^2h_n^{ -1}\]
whenever  $|x-y|\geq 3h_n$.  Therefore 
\[\frac{1}{h_n^2}L^\infty_nv_y(x)\leq \alpha |x-y|^{\alpha  -2} s_0^2\Phi (s_0)\left( (1 -\delta ')(\alpha  -1) +C|x-y|r_n^2h_n^{ -3} \right )\]
whenever  $|x-y|\geq C_{\delta '}h_n$. Combining this with \eqref{eq:L2error} gives
\[ h_n^{-2}L^p_nv_y(x)\leq C|x-y|^{\alpha  -2}\left(\alpha(p-1) + d-p +C|x-y| +C|x-y|r^2_nh_n^{ -3} +C\delta +C\delta' \right)\]
for all $y\in \X_n$ and  $x\in X_n$  with  $|x-y|\geq \max\{h_n^\sigma ,C_{\delta '}h_n\}$.   

 We now bound  $r_n$. Let  $t>0$  and partition  $\T^d$ into boxes  $R_1,R_2,R_3,\dots,R_M$ of side length  $t/\sqrt{d}$.      If  $r_n>t$ then at least one box  $R_i$   contains no points from $X_n$.    It follows that 
 \begin{align*}
 \P[r_n>t]&\leq \sum_{i=1}^M \P[X_n\cap R_i=\varnothing ]\\
 &\leq M(1 -Ct^d)^{n }\\
 &\leq C\exp\left( n\log(1 -Ct^d) -d\log (t) \right)\\
 &\leq C\exp\left(  -Cnt^d -d\log (t) \right).
 \end{align*}
 Since we assume that  $h_n\geq n^{ -1/d}$ we have that  
 \[\P[r_n^2>h_n^{3}]\leq C\exp\left(  -Cnh_n^{3d/2} +4\log (n) \right).\] 
The proof is completed by choosing  $\delta ,\delta '>0$ sufficiently small and noting that by assumption we have  $\alpha (p -1) +d -p<0$.   
\end{proof}

\begin{figure}
\centering
\includegraphics[clip=true,trim = 30 20 30 20, width=0.7\textwidth]{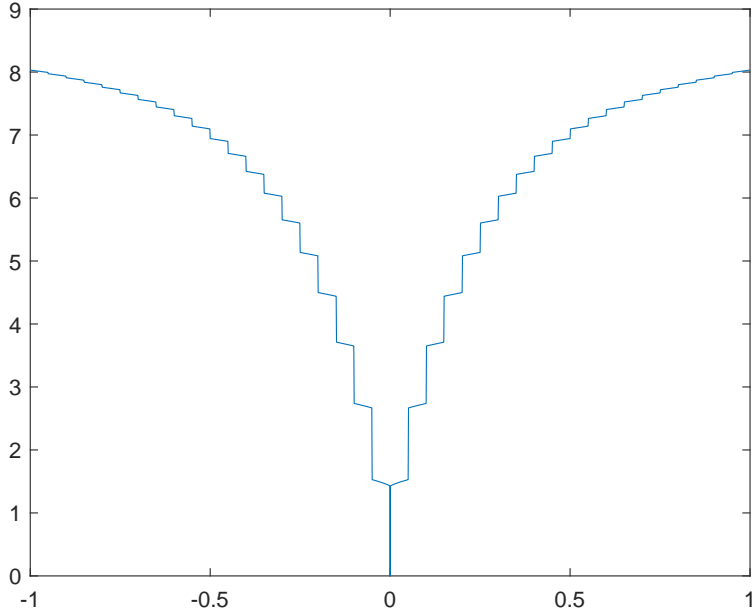}
\caption{A depiction of the local barrier, defined in \eqref{eq:ws}, used to complete the proof of H\"older regularity.}
\label{fig:barrier}
\end{figure}
In order to use the barrier function technique to prove a global H\"older regularity result, we need a barrier that is a supersolution globally, including within the local $O(h_n)$ neighborhood of $x=y$ where the barrier from Lemma \ref{lem:super} fails to be a supersolution. In Lemma \ref{lem:finesuper} we show how to construct a barrier in this local neighborhood using a construction that heavily exploits the $\min$/$\max$ structure of the $\infty$-Laplace term. Before giving the proof and the explicit form for the barrier, let us describe the idea heuristically. The new local barrier, defined in Equation \eqref{eq:ws} below and depicted in Figure \ref{fig:barrier}, consists of sharp jump discontinuities that decay in size rapidly away from the origin. The spacing between jumps is $h_n/2$ to ensure that in the definition of the $\infty$-Laplacian, the $\min$ term will take a large negative jump, while the $\max$ term (and the $2$-Laplacian term) will take a much smaller positive jump, resulting in a large negative value for $L^p_nv_y(x)$, even at points arbitrarily close to the origin. We note that this barrier does not depend on the $p>d$ assumption and works for arbitrary $p$. The sharp decay in the jump size away from the origin prevents the barrier from being a supersolution outside of this $O(h_n)$ neighborhood when $p<d$. 

We now give the statement and proof of Lemma \ref{lem:finesuper}.

\begin{lemma}\label{lem:finesuper}
Let  $p>2$ and  $0 < \alpha < 1$.  For $M>0$ and $y\in \T^d$ define
\begin{equation}\label{eq:ws}
v_y(x)=|x-y|^\alpha + Mh_n^\alpha\sum_{k=1}^{\infty}\beta^k 1_{\{2|x-y|>(k-1)h_n\}},
\end{equation}
where
\[\beta =\frac{\lambda (p -2)}{8(m\lambda (p -2) +1)} \ \ \mbox{ and } \ \ m = \max_{0\leq s \leq 2} \Phi(s).\]
Then for every $c >0$ there exists  $M>0$ such that   
\[ \P\left[\forall y\in \X_n, x\in X_n\cap B(y, (c+4)h_n)\setminus \{y\},\ L^p_nv_y(x)\leq 0\right]\geq 1-\exp\left(  -Cnh_n^{d} +2\log(n) \right).\]
\end{lemma}
\begin{proof}
Let $c>0$ and let $\Omega$ denote the event that for every $y\in \X_n$ and every $x \in B(y,(c+4)h_n)\cap X_n$  with $|x-y|> h_n$, the set
\[B_x:= B(x,h_n) \cap B(y,|x-y|-h_n/2)\]
has a nonempty intersection with $X_n$. There exists $a>0$ such that $|B_x| \geq ah_n^d$ for all $|x|> h_n$. Therefore $\P[\Omega] \geq 1-n^2e^{Cnh_n^d}$. For the rest of the proof, we assume $\Omega$ occurs.

Let $y\in \X_n$,  $x\in B(y,(c+4)h_n)\cap X_n$ and let  $k\in \N$ such that  $(k -1)h_n<2|x-y|\leq kh_n$. For  $z\in B(x,2h_n)$ we have
\begin{align*}
 v_y(z) -v_y(x)&\leq Mh^\alpha (\beta ^{k +4} +\beta ^{k +3} +\beta ^{k +2} +\beta ^{k +1}) +(|x-y| +2h_n)^\alpha  -|x-y|^\alpha \\
&\leq 4M\beta ^{k +1}h_n^\alpha  +((c+4)h_n+2h_n)^\alpha\\
&\leq 4M\beta ^{k +1}h_n^\alpha  +(c+6)^\alpha h_n^\alpha.
\end{align*}
Therefore 
\[\frac{1}{d_n(x)}L^2_nv_y(x)\leq 4M\beta ^{k +1}h_n^\alpha  +(c+6)^\alpha h_n^\alpha,\]
and
\[\max_{z\in X_n}w_n(x,z)(v_y(z) -v_y(x))\leq 4Mm\beta ^{k +1}h_n^\alpha  +m(c+6)^\alpha h_n^\alpha.\]
If $k=1,2$, then $y \in B(x,h_n)$ and since $v_y(y) = 0$ we have
\[\min_{z\in X_n}w_n(x,z)(v_y(z) -v_y(x))\leq  -M\beta ^{k}h_n^\alpha ,\]
If $k\geq 3$, then $|x-y|>h_n$. Since $B_x\cap X_n$ is nonempty, there exists $z \in B(x,h_n)\cap X_n$ such that $2|z-y| \leq (k-1)h_n$, and hence 
\[\min_{z\in X_n}w_n(x,z)(v_y(z) -v_y(x))\leq  -M\beta ^{k}h_n^\alpha.\]
Combining the above observations we have
\begin{align*}
 \hspace{0.5cm} h_n^{ -\alpha }L^p_nv_y(x)&\leq  4M\beta ^{k +1}  +(c+6)^\alpha +\lambda (p -2)\left(   4Mm\beta ^{k +1}  +m(c+6)^\alpha -M\beta ^k \right)\\
&=(c+6)^\alpha(1 +2\lambda (p -2))  -\frac{1}{2}\lambda (p-2)\beta^k M,
\end{align*}
for $x\in B(y,(c+4)h_n)\cap X_n$ with $y\neq x$.
Now we can simply choose  $M>0$ large enough so that  $L^p_nv_y(x)\leq 0$ for all  $k\leq 2c +1$.    
\end{proof}

We are now ready to give the proof of Theorem \ref{thm:holder}. The proof involves patching together the barrier functions provided in Lemmas \ref{lem:super} and \ref{lem:finesuper}, and using the global barrier to establish H\"older regularity.

\begin{proof}[Proof of Theorem \ref{thm:holder}]
Choose  $r,c,\delta>0$ so that the conclusions of Lemma \ref{lem:super} hold for  $\sigma =1$. We may assume without loss of generality that  $c$ is a positive integer.   Define  $\beta$ as in Lemma \ref{lem:finesuper} and let  $M>0$ so that the conclusions of Lemma \ref{lem:finesuper}  hold with the aforementioned value of $c$.   For the rest of the theorem, we assume the conclusions of Lemmas \ref{lem:super} and \ref{lem:finesuper} hold, which is an event with probability at least $1-\exp(-\delta nh_n^q + C\log(n))$, where $q=\max\{d+4,3d/2\}$.

The proof is now split into three steps.  

1. Let $y\in \X_n$,  define  $k_0=2c +5$ and set 
\[ v_1(x)=|x-y|^\alpha  +Mh_n^\alpha \sum_{k=1}^\infty \beta ^k1_{\{2|x -y|>(k -1)h_n \}},\]
\[ v_2(x)=|x-y|^\alpha  +\frac{M\beta h_n^\alpha }{1 -\beta }(1 -\beta ^{k_0}),\]
and
\[ v_y(x)=\min\left\{ v_1(x),v_2(x) \right\}.\]
We claim that $L^p_nv_y(x) \leq 0$ for all $0<|x-y| \leq r$. To see this, 
notice that for  $|x -y|\leq (c +2)h_n$ we have $v_y(x)=v_1(x)$ and for  $|x -y|\geq (c +2)h_n$ we have  $v_y(x)=v_2(x)$. By Lemmas \ref{lem:super} and \ref{lem:finesuper} we see that  $L^p_nv_1(x)\leq 0$ whenever  $|x -y|\leq (c +4)h_n$  and  $L^p_nv_2(x)\leq 0$   whenever $|x -y|\geq ch_n$. It follows that  $L^p_nv_y(x)\leq 0$ for all  $x\neq  y$  such that $r\geq |x -y|\geq (c +4)h_n$ or  $|x -y|\leq ch_n$.     If  $(c +2)h_n\leq |x -y|\leq (c +4)h_n$ then  $v_y(x)=v_2(x)$ and we have  $L^p_nv_y(x)\leq L^p_nv_2(x)\leq 0$ because $v\leq v_2$ in $B(x,2h_n)$. Likewise, if $ch_n \leq |x-y|\leq (c+2)h_n$ then $L^p_nv_y(x)\leq L^p_nv_1(x) \leq 0$. This establishes the claim.

2. Let $y\in \O$. Let $C>0$ be large enough so that for all $x,y\in \O$
\[ Cv_y(x) \geq \max_\O g - \min_\O g.\]
 It follows that
\[g(y) -Cv_y(x) \leq u_n(x) \leq g(y) + Cv_y(x)\]
holds for all $x\in \O$. By the maximum principle (Theorem \ref{thm:maxprinc}) and the fact that $u_n(y)=g(y)$ we have that 
\begin{equation}\label{eq:gkey}
u_n(y) -Cv_y \leq u_n \leq u_n(y) + Cv_y
\end{equation}
for all $y\in \O$.
   
3.  Let  $y\in X_n$. We claim that
\[u_n(y) - Cv_y(z) \leq u_n(z) \leq u_n(y) + Cv_y(z)\]
holds for all $z\in \O$. Indeed, recalling \eqref{eq:gkey} we have 
\[u_n(y) \geq u_n(z) - Cv_z(y) = u_n(z) - Cv_y(z),\]
and
\[u_n(y) \leq u_n(z) + Cv_z(y) = u_n(z) +Cv_y(z),\]
which establishes the claim.
By the maximum principle (Theorem \ref{thm:maxprinc}) we have that 
\begin{equation}\label{eq:ukey}
u_n(y) -Cv_y \leq u_n \leq u_n(y) + Cv_y
\end{equation}
for all $y\in X_n$.
   
4. Notice that 
\begin{equation}\label{eq:vy}
|x -y|^\alpha \leq v_y(x)\leq |x -y|^\alpha  +\frac{M\beta h_n^\alpha}{1-\beta}.
\end{equation}
Combining this with \eqref{eq:gkey} and \eqref{eq:ukey} we have 
\begin{equation}\label{eq:hg2}
u_n(y) -C(|x -y|^\alpha  +h^\alpha _n)\leq u_n(x)\leq u_n(y) +C(|x -y|^\alpha  +h^\alpha _n), 
\end{equation}
for all $x,y\in \X_n$, which completes the proof.
\end{proof}

\section{Convergence proof}
\label{sec:conv}

In this section we prove Theorem \ref{thm:main}. We first recall the notion of viscosity solution for
\begin{equation}\label{eq:pde}
F(\nabla^2 u,\nabla u,x)=0\ \ \mbox{ in } \Omega,
\end{equation}
 where $\Omega\subset \R^d$  is open. 
\begin{definition}\label{def:viscosity}
We say that  $u\in C(\Omega)$ is a viscosity subsolution of   \eqref{eq:pde} if for every  $x_0\in \Omega$ and $\varphi \in C^\infty(\Omega)$ such that  $u -\varphi $  has a local maximum at $x_0$ we have 
\[F(\nabla^2 \phi(x_0),\nabla \phi(x_0),x_0)\leq 0.\]      
We say that  $u\in C(\Omega)$ is a viscosity supersolution of \eqref{eq:pde} if for every  $x_0\in \Omega$ and $\varphi \in C^\infty(\Omega)$ such that  $u -\varphi $  has a local minimum at $x_0$ we have 
\[F(\nabla^2 \phi(x_0),\nabla \phi(x_0),x_0)\geq 0.\]      
We say that  $u\in C(\Omega)$ is a viscosity solution of \eqref{eq:pde} if  $u$ is both a viscosity sub- and supersolution of \eqref{eq:pde}. 
\end{definition}
Let  $\pi :\R^d\to \T^d$  be the projection operator.
\begin{definition}\label{def:viscosity2}
A function  $u\in C(\T^d)$  is a viscosity solution of \eqref{eq:pLap} if $u=g$  on $\O$ and $v(x):=u(\pi (x))$ is a viscosity solution of \eqref{eq:pde} with      
\[F_p(X,q,x) := -|q|^{p-2}(\mbox{Trace}(X) + 2\nabla \log f(x) \cdot q + (p-2)|q|^{-2}q\cdot X q),\]
and $\Omega:=\pi ^{ -1}(\T^d\setminus \O)$.
\end{definition}
We first verify uniqueness of viscosity solutions of \eqref{eq:pLap}.
\begin{theorem}\label{thm:unique}
Assume $p>d \geq 2$ and $f\in C^{2}(\T^d)$  is positive. Then there exists a unique viscosity solution of \eqref{eq:pLap}. 
\end{theorem}
The proof of Theorem \ref{thm:unique} follows \cite{juutinen2001equivalence} closely, with some minor adjustments to handle the weighted $p$-Laplacian, and some simplifications due to our assumption that $p>d \geq 2$. In particular, the proof shows that weak and viscosity solutions of \eqref{eq:pLap} coincide.
\begin{proof}
Existence follows from the proof of Theorem \ref{thm:main}. We only need to prove uniqueness here.

Let $u\in C(\T^d)$ be a viscosity solution of \eqref{eq:pLap}. Extending $u$ to $\R^d$ by setting $u(x)=u(\pi(x))$, we have that $u\in C(\R^d)$ is a $\Z^d$-periodic viscosity solution of 
\begin{equation}
\left.
\begin{array}{rll}
F_p(\nabla^2 u, \nabla u, x) &= 0&\mbox{in } \Omega\\
u&=g&\mbox{on } \partial \Omega,\\
\end{array}
 \right\}
\label{eq:pLapRd}
\end{equation}
where $\Omega = \pi^{-1}(\T^d\setminus \O)$, $f(x)=f(\pi(x))$ and $g(x)=g(\pi(x))$. For $\eps \geq 0$, let $u_\eps\in W^{1,p}_{per}(\R^d)$ be the unique $\Z^d$-periodic weak solution (defined via integration by parts) of
\begin{equation}
\left.
\begin{array}{rll}
-\div\left(f^2|\nabla u_\eps|^{p-2}\nabla u_\eps\right) &= \eps&\mbox{in } \Omega\\
u_\eps&=g&\mbox{on } \partial \Omega.\\
\end{array}
 \right\}
\label{eq:pLapeps}
\end{equation}
The solutions $u_\eps$ can be constructed by the Calculus of Variations, for instance.
By regularity theory for degenerate elliptic equations \cite{tolksdorf1984regularity,lieberman1988boundary}, we have $u_\eps \in C^{1,\gamma}_{loc}(\Omega)$, and since $p>d$ we have $u_\eps\in C^{0,1-d/p}(\R^d)$ by Morrey's inequality. 

We will show that $u=u_0$. The proof is split into 4 steps.

1. First, we claim that for $\eps>0$, $u_\eps$ is a viscosity supersolution of 
\begin{equation}
f(x)^2F_p(\nabla^2 u_\eps, \nabla u_\eps, x) = \eps\ \  \mbox{ in } \ \  \Omega.
\label{eq:pLape}
\end{equation}
To see this, let $x_0 \in \Omega$ and $\phi \in C^\infty(\Omega)$ such that $u_\eps-\phi$ has a local minimum at $x_0$. We need to show that
\[-\div\left(f^2|\nabla \phi|^{p-2}\nabla \phi\right)\Big\vert_{x=x_0} \geq \eps.\]
Assume to the contrary that
\[-\div\left(f^2|\nabla \phi|^{p-2}\nabla \phi\right)\Big\vert_{x=x_0}< \eps.\]
 For $\delta>0$ define
\[\psi(x) = \phi(x) + \delta^2 +u_\eps(x_0)-\phi(x_0) - \delta|x-x_0|^4. \]
For sufficiently small $r,\delta>0$ we have $u_\eps(x) \geq \psi(x)$  for all $ x\in \partial B(x_0,r)$, and
\begin{equation}\label{eq:cont}
-\div\left(f^2|\nabla \psi|^{p-2}\nabla \psi\right)\leq \eps \ \  \mbox{ in } \ \  B(x_0,r).
\end{equation}
 The comparison principle for weak solutions of the $p$-Laplace equation yields $\psi \leq u_\eps$ in $B(x_0,r)$, which is a contradiction to the fact that $\psi(x_0)>u_\eps(x_0)$. This establishes the claim.


2. We now show that $u \leq u_\eps$ for $\eps>0$. Fix $\eps>0$ and for $\delta>0$ define
\[\Omega_\delta = \{x\in \Omega \, : \, \mbox{dist}(x,\partial \Omega)> \delta\}.\]
Let $M>0$ such that $|\nabla u_\eps|\leq M$ on $\Omega_\delta$ and define the truncated operator
\[ \bar{F}_p(X,q,x):= -\min\{M,|q|\}^{p-2}(\mbox{Trace}(X) + 2\nabla \log f(x) \cdot q + (p-2)|q|^{-2}q\cdot X q).\]
Then $u$ is clearly a viscosity solution of $\bar{F}_p=0$ in $\Omega$, and $u_\eps$ is a viscosity solution of $f^2\bar{F}_p \geq \eps$ in $\Omega_\delta$. Since $f\geq \theta > 0$ for some $\theta>0$, and $\bar{F}_p$ satisfies
\[\bar{F}_p(y,q,Y) - \bar{F}_p(x,q,X) \leq 2 M^{p-2}[\nabla f]_{0,1}|q| |x-y|,\]
for all $q\in \R^d$, $x,y\in \Omega_\delta$ and symmetric matrices $X\leq Y$, we can invoke the comparison principle for viscosity solutions~\cite{crandall1992user} to obtain 
\[\max_{\partial \Omega_\delta} (u - u_\eps) = \max_{\bar{\Omega_\delta}} (u - u_\eps). \]
Since $u,u_\eps$ are uniformly continuous on $\R^d$, we can send $\delta\to 0^+$ to find that $u \leq u_\eps$ on $\R^d$.

3. We now send $\eps\to 0^+$ to show that $u\leq u_0$. Using $\phi=\max\{u_\eps-u_0-\delta,0\}$ as a test function in the definition of weak solution of \eqref{eq:pLapeps} for $u_\eps$ and $u_0$ and subtracting, we have
\[ \int_{A_\delta} f^2|\nabla u_\eps - \nabla u_0|^p \, dx \leq C\int_{A_\delta} f^2(|\nabla u_\eps|^{p-2}\nabla u_\eps - |\nabla u_0|^{p-2}\nabla u_0) \cdot (\nabla u_\eps - \nabla u_0) \, dx \leq C\eps,\]
where $A_\delta = \{ x \in (0,1)^d \, : \, u_\eps(x) - u_0(x) > \delta\}$ and $\delta >0$. Sending $\delta\to 0^+$ we find that
\[\int_{\{u_\eps > u\}} f^2|\nabla u_\eps - \nabla u_0|^p \, dx \leq C\eps.\]
A similar argument gives
\[\int_{\{u_\eps < u\}} f^2|\nabla u_\eps - \nabla u_0|^p \, dx \leq C\eps.\]
By Morrey's inequality $u_\eps \to u_0$ uniformly on $\R^d$ as $\eps\to 0^+$. Therefore $u\leq u_0$.

4. To see that $u\geq u_0$, we apply the argument above to $v:=-u$ to find that $v\leq -u_0$, or $u \geq u_0$. Therefore $u=u_0$, which completes the proof.
\end{proof}

We now give the proof of our main result. 
\begin{proof}[Proof of  Theorem \ref{thm:main}]
The proof is split into several steps.

1. Since 
\[\lim_{n\to \infty} \frac{nh_n^q}{\log(n)}=\infty,\]
we can apply the Borel-Cantelli Lemma and Theorems \ref{thm:holder}, \ref{thm:consistency2}, \ref{thm:liconsistency}, and Lemma \ref{lem:approximation1} to show that with probability one
\begin{equation}\label{eq:Lpcon}
\lim_{\substack{n\to \infty\cr x\to x_0}} \frac{1}{h_n^2}L^p_n\phi(x) = \frac{C_2}{C_1}f^{-2}|\nabla \phi|^{2-p} \mbox{div}(f^2|\nabla \phi|^{p-2} \nabla \phi)\big\vert_{x=x_0}
\end{equation}
for all $x_0\in \T^d\setminus \O$ and $\phi \in C^\infty(\R^d)$ with $\nabla \phi(x_0)\neq 0$, and 
\begin{equation}\label{eq:holder2}
|u_n(x)-u_n(y)|\leq C(|x-y|^{1-d/p} + h_n^{1-d/p})
\end{equation}
for $n$ sufficiently large. For the remainder of the proof, we fix a realization in this probability one event.

2. Let $j_n:\T^d \to \X_n$ be the closest point projection, which satisfies
\[|x-j_n(x)| = \min_{y\in \X_n}|x -y|.\]
Define  $v_n:\T^d\to \R$ by  $v_n(x)=u_n(j_n(x))$.   
Since  $|x -j_n(x)|\leq r_n$, where $r_n$ is defined in \eqref{eq:r}, it follows from \eqref{eq:holder2} that for any  $x,y\in \T^d$  
\begin{align*}
|v_n(x) -v_n(y)|&\leq C(|j_n(x) -j_n(y)|^{1-d/p} +h_n^{1-d/p})\\
&=C(|j_n(x) -x +y -j_n(y) +x -y|^{1-d/p} +h_n^{1-d/p})\\
&\leq C(|x -y|^{1-d/p} +2r_n^{1-d/p} +h_n^{1-d/p}).
\end{align*}
Since  $r_n,h_n\to 0$ as  $n\to \infty$, we can use the Arzel\`a-Ascoli Theorem (see the appendix in \cite{calder2015PDE}) to extract a subsequence, denote again by  $v_n$, and a H\"older continuous function  $u\in C^{0,1-d/p}(\T^d)$ such that  $v_n\to u$ uniformly on  $\T^d$ as  $n\to \infty$. Since  $u_n(x)=v_n(x)$ for all  $x\in \X_n$ we have      
\begin{equation}\label{eq:uniform}
\lim_{n\to \infty}\max_{x\in \X_n}|u_n(x) - u(x)|=0.  
\end{equation}
We claim that  $u$ is the unique viscosity solution of \eqref{eq:pLap}, which will complete the proof. 

3. We first show that  $u$ is a viscosity subsolution  of \eqref{eq:pLap}. Let $x_0\in \T^d\setminus \O$ and $\varphi \in C^\infty(\R^d)$  such that  $u -\varphi $  has a strict global maximum at the point $x_0$ and $\nabla  \varphi (x_0)\neq  0$. We need to show that
\[\mbox{div}(f^2|\nabla \phi|^{p-2} \nabla \phi)\big\vert_{x=x_0}\geq 0.\]
By \eqref{eq:uniform} there exists a sequence of points  $x_n\in \X_n$   such that $u_n -\varphi $   attains its global maximum at $x_n$ and  $x_n\to x_0$ as $n\to \infty$.          Therefore 
\[u_n(x_n) -u_n(x)\geq \varphi (x_n) -\varphi (x)\quad \mbox{  for all }x\in \X_n.\] 
Since  $x_0 \not\in \O$, we have that  $x_n \not\in \O$ for  $n$ sufficiently large and so $L^p_nu_n(x_n) \leq L^p_n \varphi(x_n)$. By \eqref{eq:Lpcon} we have
\[ 0=\lim_{n\to \infty}\frac{1}{h_n^2} L^p_nu_n(x_n)\leq\lim_{n\to \infty} \frac{1}{h_n^2} L^p_n\varphi (x_n)= \frac{C_2}{C_1}f^{-2}|\nabla \phi|^{2-p} \mbox{div}(f^2|\nabla \phi|^{p-2} \nabla \phi)\big\vert_{x=x_0}.\]
 Thus $u$ is a viscosity subsolution of \eqref{eq:pLap}. 

4. To verify the supersolution property, set  $v_n= -u_n$  and note that $L^p_nv_n= -L^p_nu_n=0$ and  $v_n\to  -u$  uniformly as $n\to \infty$.   The subsolution argument above shows that  $ -u$ is a viscosity subsolution of \eqref{eq:pLap}, and hence   $u$ is a viscosity supersolution.   This completes the proof.
\end{proof}

\section*{Acknowledgments}

The author gratefully acknowledges the support of NSF-DMS grant 1713691. The author is also grateful to the anonymous referees, whose suggestions have greatly improved the paper.

\appendix

\section{Consistency for the graph $p$-Laplacian}

Consider the graph $p$-Laplacian
\begin{equation}\label{eq:}
\Delta^G_p u(x) := \sum_{y\in \X_n}w_n(x,y)|u_n(x) -u_n(y)|^{p -2}(u_n(x) -u_{n}(y)).
\end{equation}
It is possible to prove Theorem \ref{thm:main} for the graph $p$-Laplacian provided $nh_n^p \to 0$ as $n\to \infty$. The proof is very similar to Theorem \ref{thm:main}; the main difference is the consistency result, which we sketch here for completeness.

 We require an integration lemma from~\cite{ishii2010class}.
\begin{lemma}\label{lem:int}
Let $\rho:[0,\infty)\to [0,\infty)$ be a nonnegative continuous function with compact support, and let $p_1,\dots,p_d$ be real numbers equal or larger than $1/2$. Then
\[ \int_{\R^d} \rho(|x|^2)|x_1|^{2p_1-1}\cdots |x_d|^{2p_d-1} \, dx = \frac{\Gamma(p_1)\cdots \Gamma(p_d)}{\Gamma(p_1+\cdots +p_d)} \int_0^\infty \rho(t) t^{p_1 +\cdots +p_d-1}\, dt,\]
where $\Gamma$ denotes the Gamma function 
\[\Gamma(t) = \int_0^\infty e^{-x}x^{t-1}\, dx.\]
\end{lemma}
We prove here consistency in expectation. Arguments similar to the proof of Theorem \ref{thm:consistency2} can be used to obtain consistency with high probability and almost surely as $n\to \infty$.
\begin{theorem}\label{thm:consistency_p}
Let $\phi \in C^3(\R^d)$ and $p\geq 3$. Then
\begin{equation}\label{eq:consistency_p}
\E[\Delta^G_p \phi(x)]=\frac{1}{2}c(p,d)f^{-1}\mbox{\normalfont div}(f^2|\D \phi|^{p-2}\D \phi)nh^{d+p} + R(x)nh^{d+p+1}, 
\end{equation}
where
\begin{equation}\label{eq:R}
|R(x)| \leq C(c_1^{p-2}(c_1+c_2+c_3) + c_1^{p-3}c_2^3h^2) \ \ \ \mbox{for all } x\in \R^d,
\end{equation}
and $c_k = \|D^k\phi\|_{L^\infty(\R^d)}$.
\end{theorem}
\begin{proof}
We may assume $x=0$ and $\phi(0)=0$. We write $h=h_n$ for simplicity. Then we have
\[\E[\dJ \phi(0)] = n\int_{\R^d} \Phi_h\left(|y|\right)^p|\phi(y)|^{p-2}\phi(y)f(y) \, dy.\]
Set $x = y/h$ to find that
\[\E[\dJ \phi(0)] = nh^d\int_{\R^d} \Phi(|x|)^p\Psi(\phi(xh))f(xh) \, dx,\]
where
\[\Psi(t) = |t|^{p-2}t.\]
 Then
\begin{equation}
\Psi(t) = \Psi(a) + \Psi'(a)(t-a) + O(T^{p-3}|t-a|^3) \ \ \ \mbox{for } a,t \in [-T,T].
\label{eq:PsiTS}
\end{equation}
Letting $q=\D \phi(0)$ and $Q = \nabla^2\phi(0)$ we have
\begin{equation}
\phi(xh) = O(c_1|x|h),
\label{eq:phiTS0}
\end{equation}
\begin{equation}
\phi(xh) = hq \cdot x + O(c_2|xh|^2),
\label{eq:phiTS1}
\end{equation}
and
\begin{equation}\label{eq:phiTS}
\phi(xh) = hq\cdot x + \frac{h^2}{2}x\cdot Qx + O(c_3|xh|^3),
\end{equation}
Therefore
\begin{align*}
 \hspace{1cm}\Psi(\phi(xh)) &= \Psi(hq\cdot x) + \Psi'(hq\cdot x)(\phi(xh)-hq\cdot x) + O(c_1^{p-3}|xh|^{p-3}c_2^3|xh|^6)\\
&=\Psi(hq\cdot x) + \Psi'(hq\cdot x)\left(\frac{h^2}{2}x\cdot Qx + O(c_3|x|^3h^3)\right) + O(c_1^{p-3}c_2^3|xh|^{p+3})\\
&=\Psi(hq\cdot x) + \frac{h^2}{2}\Psi'(hq\cdot x)x\cdot Qx + O(c_1^{p-2}c_3|xh|^{p+1} + c_1^{p-3}c_2^3|xh|^{p+3}),
\end{align*}
and we have
\begin{align*}
 \Psi(\phi(xh))f(xh)&\\
&\hspace{-2cm}= f(0)\Psi(hq\cdot x) + f(0)\frac{h^2}{2}\Psi'(hq\cdot x)(x\cdot Qx) + (f(xh)-f(0))\Psi(hq\cdot x) \\
&\hspace{-1.5cm}+ O\left( c_1^{p-2}(c_2+c_3)|xh|^{p+1} + c_1^{p-2}c_3|xh|^{p+2} + c_1^{p-3}c_2^3|xh|^{p+3} + c_1^{p-3}c_2^3|xh|^{p+4} \right)\\
&\hspace{-2cm}= f(0)\Psi(hq\cdot x) + f(0)\frac{h^2}{2}\Psi'(hq\cdot x)(x\cdot Qx) + h(\nabla f(0)\cdot x)\Psi(hq\cdot x) \\
&\hspace{-1.5cm}+ O\left( c_1^{p-2}(c_1+c_2+c_3)|xh|^{p+1} + c_1^{p-2}c_3|xh|^{p+2} + c_1^{p-3}c_2^3|xh|^{p+3} + c_1^{p-3}c_2^3|xh|^{p+4} \right).
\end{align*}
Integrating against $\Phi(|x|)^p$, the first term is odd and vanishes, so we get
\begin{align*}
&\int_{\R^d}\Phi(|x|)^p|\phi(xh)|^{p-2}\phi(xh)f(xh)\, dx\\
&=f(0)h^p\int_{\R^d} \Phi(|x|)^p\left( (\tfrac{1}{2}x\cdot Qx)\Psi'(q\cdot x) + (\nabla \log f(0)\cdot x) \Psi(q\cdot x) \right)\, dx\\
&\hspace{8mm}+ O\left(c_1^{p-2}(c_1+c_2+c_3)h^{p+1} + c_1^{p-3}c_2^3h^{p+3}\right).
\end{align*}
Therefore
\begin{align}
 \frac{\E[\dJ\phi(0)] }{nh^{d+p}f(0)} &= \int_{\R^d} \Phi(|x|)^p\left( (\tfrac{1}{2}x\cdot Qx)\Psi'(q\cdot x) + (\nabla \log f(0)\cdot x) \Psi(q\cdot x) \right)\, dx \nonumber \\
&\hspace{1.5in} + O\left(c_1^{p-2}(c_1+c_2+c_3)h + c_1^{p-3}c_2^3h^{3}\right)\nonumber \\
&= \int_{\R^d} \Phi(|x|)^p|q\cdot x|^{p-2}\left( \tfrac{1}{2}(p-1)(x\cdot Qx) + (\nabla \log f(0)\cdot x)(q\cdot x) \right)\, dx \nonumber \\
&\hspace{1.5in} + O\left(c_1^{p-2}(c_1+c_2+c_3)h + c_1^{p-3}c_2^3h^{3}\right)\nonumber\\
&=A+B+O\left(c_1^{p-2}(c_1+c_2+c_3)h + c_1^{p-3}c_2^3h^{3}\right),
\label{eq:Eint}
\end{align}
where
\begin{equation}\label{eq:Ahere}
A = \frac{p-1}{2}\sum_{i,j=1}^d Q_{ij}\int_{\R^d} \Phi(|x|)^p |q\cdot x|^{p-2}x_ix_j \, dx,
\end{equation}
and
\begin{equation}\label{eq:Bhere}
B = \sum_{i=1}^d(\log f)_{x_i}\int_{\R^d}\Phi(|x|)^p |q\cdot x|^{p-2} (q\cdot x) x_i\, dx.
\end{equation}

Let us tackle $A$ first. We may assume that $q\neq 0$. Let $O$ be an orthogonal transformation so that $Oe_d = q/|q|$. In particular $O_{id} = [Oe_d]_i= q_i/|q|$. Making the change of variables $y=O^Tx$ we have
\begin{align*}
B &= \sum_{i=1}^d(\log f)_{x_i}\int_{\R^d}\Phi(|y|)^p ||q|y_d|^{p-2}(|q|y_d) [Oy]_i\, dy\\
&= |q|^{p-1}\sum_{i,j=1}^d (\log f)_{x_i} O_{ij}\int_{\R^d}\Phi(|y|)^p |y_d|^{p-2} y_dy_j\, dy.
\end{align*}
If $j\neq d$ in the sum above, then the integral vanishes, since the integrand is odd. Therefore
\begin{equation}\label{eq:Bfinal}
 B= |q|^{p-1}\sum_{i=1}^d (\log f)_{x_i} O_{id}\int_{B_2}\Phi(|y|)^p |y_d|^p\, dy= c(p,d)|q|^{p-2} (\D \log f)\cdot q,
\end{equation}
where
\[c(p,d) =  \int_{\R^d}\Phi(|x|)^p |x_d|^p\, dx.\]

For $A$, we again make the change of variables $y=O^Tx$. Then  we have
\begin{align*}
A &= \frac{p-1}{2}|q|^{p-2}\sum_{i,j=1}^d Q_{ij}\sum_{k,\ell=1}^d O_{ik}O_{j\ell}\int_{B_2} \Phi(|y|)^p |y_d|^{p-2}y_ky_\ell \, dy\\
&= \frac{p-1}{2}|q|^{p-2}\sum_{k=1}^d[O^TQO]_{kk}\int_{B_2} \Phi(|y|)^p |y_d|^{p-2}y_k^2 \, dy\\
&= \frac{p-1}{2}|q|^{p-2}\left( \sigma(p,d)\mbox{Tr}(Q) + (c(p,d)-\sigma(p,d))|q|^{-2}q\cdot Qq\right),
\end{align*}
where
\[\sigma(p,d) = \int_{\R^d} \Phi(|x|)^p |x_d|^{p-2} x_k^2\, dx\]
for any $k < d$. By Lemma \ref{lem:int} we have
\[\sigma(p,d) = \frac{\Gamma\left(\frac{1}{2}\right)^{d-2}\Gamma\left(\frac{p-1}{2}\right)\Gamma\left(\frac{3}{2}\right)}{\Gamma\left(\frac{d+p}{2}\right)}\int_0^\infty \Phi(\sqrt{t})t^{\frac{d+p}{2}-1}\, dt,\]
and
\[c(p,d) = \frac{\Gamma\left(\frac{1}{2}\right)^{d-1}\Gamma\left(\frac{p+1}{2}\right)}{\Gamma\left(\frac{d+p}{2}\right)}\int_0^\infty \Phi(\sqrt{t})t^{\frac{d+p}{2}-1}\, dt.\]
Using the identity $\Gamma(t+1)=t\Gamma(t)$ we find that
\[c(p,d) = (p-1)\sigma(p,d).\]
It follows that
\[A = \frac{1}{2}c(p,d)|q|^{p-2}(\mbox{Tr}(Q) + (p-2)|q|^{-2} q^TQq^T).\]

Combining this with \eqref{eq:Eint} and \eqref{eq:Bfinal} yields
\begin{align*}
&\frac{\E[\dJ\phi(0)] }{nh^{d+p}} \\
&= \frac{1}{2}c(p,d)f|\nabla \phi|^{p-2}\left( \Delta \phi + 2\nabla \log f\cdot \nabla \phi + (p-2)\Delta_\infty \phi \right)\\
&\hspace{2in}+O\left(c_1^{p-2}(c_1+c_2+c_3)h + c_1^{p-3}c_2^3h^{3}\right)\\
&=\frac{1}{2}c(p,d)f^{-1}\div(f^2|\D \phi|^{p-2}\D \phi) + O\left(c_1^{p-2}(c_1+c_2+c_3)h + c_1^{p-3}c_2^3h^{3}\right).
\end{align*}

If $q=0$ then from \eqref{eq:Eint} we have
\[\frac{\E[\dJ\phi(0)] }{nh^{d+p}} = O\left(c_1^{p-2}(c_1+c_2+c_3)h + c_1^{p-3}c_2^3h^{3}\right).\]
This completes the proof.
\end{proof}

\end{document}